\documentclass[11pt,twoside]{article}

\usepackage{graphicx}
\usepackage{amsfonts}
\usepackage{amsmath}
\usepackage{amssymb}
\usepackage{indentfirst}
\usepackage{booktabs}
\usepackage{verbatim}
\usepackage{color}
\usepackage{amsthm}

\usepackage{pdflscape}

\usepackage{epstopdf}

\usepackage[page,header]{appendix}

\numberwithin{equation}{section}
\newtheorem{theorem}{Theorem}[section]

\newtheorem{proposition}[theorem]{Proposition}
\newtheorem{definition}[theorem]{Definition}
\newtheorem{remark}[theorem]{Remark}

\begin{document}

	\title{A structure preserving numerical scheme for Fokker-Planck equations of neuron networks: numerical analysis and exploration }
	
	\author{Jingwei Hu\footnote{Department of Mathematics, Purdue University, West Lafayette, IN 47907, USA (jingweihu@purdue.edu).} \  \
	        Jian-Guo Liu\footnote{Department of Mathematics and Department of Physics, Duke University, Box 90320, Durham, NC
   27708, USA (jliu@phy.duke.edu).} \   \
	        Yantong Xie\footnote{School of Mathematics Science, Peking University, Beijing, 100871, China (1901110045@pku.edu.cn).} \   \
	        Zhennan Zhou\footnote{Beijing International Center for Mathematical Research, Peking University, Beijing, 100871, China (zhennan@bicmr.pku.edu.cn).}}
	\maketitle

\begin{abstract}

In this work, we are concerned with the Fokker-Planck equations associated with the Nonlinear Noisy Leaky Integrate-and-Fire model for neuron networks. Due to the jump mechanism at the microscopic level, such Fokker-Planck equations are endowed with an unconventional structure: transporting the boundary flux to a specific interior point. While the equations exhibit diversified solutions from various numerical observations, the properties of solutions are not yet completely understood, and by far  there has been no rigorous numerical analysis work concerning such models. We propose a conservative and {conditionally positivity preserving} scheme for these Fokker-Planck equations, and we show that in the linear case, the semi-discrete scheme satisfies the discrete relative entropy estimate, which essentially matches the only known long time asymptotic solution property.  We also provide extensive numerical tests to verify the scheme properties, and carry out several sets of numerical experiments, including finite-time blowup, convergence to equilibrium and capturing time-period solutions of the variant models.

\end{abstract}

{\small

}


\section{Introduction}


As large-scale neuron networks models in computational neuroscience have received more attention \cite{CCP2011,CS2018,CP2014,CCT2010}, the need of  developing mathematical tools for analyzing the dynamics of large-scale networks and for robust numerical simulations  becomes urgent. Among various mathematical models, investigating the stochastic integrate-and-fire model for the membrane potential across a neuron has long been an active field. Through the mean-field theory \cite{NS2018,DG2017}, one can approximate the specific pattern of the neuron with the average input of the network and then derive an effective stochastic differential equation (SDE) for a single neuron \cite{ CS2018,BH1999,THF2012,DG2017}. In the past few years, more mathematicians are interested in the Fokker-Planck type equation associated with  such  SDEs (see \cite{CCP2011} for a summary, and the references therein).

One of the most well-established stochastic models in this area  is the so-called Nonlinear Noisy Leaky Integrate-and-Fire (NNLIF) model for neuron networks. The model describes the mean-field dynamical behaviour of an ensemble of neurons within a network through a stochastic differential equation for the membrane potential evolution \(V(t)\). When the firing event does not occur, the potential $V(t)$ is influenced by a relaxation to a resting potential \(V_L\) which is assumed to be independent of other neurons
and an incoming synaptic current \(I(t)\) from other neurons and the background. In many literatures, the current \(I(t)\) is approximately decomposed into  a drift term and a term of Brownian motion. The SDE, in the simplest form, is given by
\begin{equation} \label{eq:sde}
dV=-(V-V_L)dt+\mu dt+\sigma dB_t.
\end{equation}
The parameters \(\mu\) and \(\sigma\) are determined by the input current \(I(t)\).

The distinguished feature of this model is the incorporation of the firing event:  when a neuron reaches a firing potential \(V_F\), it discharges itself and the potential jumps to a resetting potential \(V_R\) immediately. Here, we assume, \(V_L<V_R<V_F\), and the firing event is expressed by
\begin{equation}\label{eq:jump}
V(t^-)=V_F, \quad V(t^+)=V_R.
\end{equation}
Equation (\ref{eq:sde}) and (\ref{eq:jump}) constitutes the stochastic process for the NNLIF model, which is a SDE coupled with a renewal process.

There are quite a few mathematical studies of the NNLIF model as in the simplest form (see, e.g. \cite{CCP2011,CGGS2013,DIRT2015,PCSS2015}). In fact, people have proposed a number of modifications to the model in order to match more complicated biological phenomenon. In \cite{CS2018,NT2000}, the network is divided into two population (excitatory and inhibitory) coupled via synapses in which the conductance of a neuron is modulated according to the firing rates of its presynaptic populations. Some studies focus on  the NNLIF model with transmission delays between the neurons and  the neurons remain in a refractory state for a certain time \cite{CS2018,CP2014,CRSS2018}. In \cite{DHT2016}, the authors consider the effect of the time passed since the neuron's last firing, which is interpreted as the age of a neuron. In \cite{CCT2010}, a generalized  Fokker-Planck equation has been studied, where the density function \(p(v,g,t)\) stands for the probability of  finding a neuron with potential \(v\) and conductance \(g\) at time \(t\).

Albeit numerous existing variants, in this article we only consider the NNLIF model as in Equation \eqref{eq:sde} and \eqref{eq:jump}. Heuristically, by Ito's formula, one can derive the time evolution of density function $p(v,t)$, which represents the probability of finding a neuron at voltage $v$ and given time $t$. The resulting PDE is a Fokker-Planck type equation with a flux shift, given in the following
\begin{equation} \label{eq:fk}
\begin{cases}
\partial_t p+\partial_v ( h p) - a \partial_{vv}p
=0, \quad v\in(-\infty,V_F)/\{V_R\},  \\
p(v,0)=p_0(v),\qquad p(-\infty,t)=p(V_F,t)=0, \\
p(V_R^-,t)=p(V_R^+,t),\qquad \partial_vp(V_R^-,t)=\partial_v p(V_R^+,t)+\frac{N(t)}{a},
\end{cases}
\end{equation}
where $p_0(v)$ is the initial condition satisfying
\[\int_{-\infty}^{V_F} p_0(v) dv=1,\]
 $h$ denotes the drift field, $a$ denotes the diffusion coefficients (which, for simplicity, is assumed to be independent of $v$), and $N(t)$ represents the mean firing rate, which takes the following form
 \begin{equation} \label{eq:Nt}
N(t)=-a\partial_v p(V_F,t) \ge 0.
\end{equation}
By direct calculation, it is clear that the choice of the mean firing rate ensures \(\int_{-\infty}^{V_F}p(v,t)dv=1\) at any time $t \ge 0$.


To match with microscopic stochastic model \eqref{eq:sde}, the drift and the diffusion parameters in the Fokker-Planck equation are taken as \(h=-(v-V_L)+\mu\) and \(a=\frac{\sigma^2}{2}\). In quite a few recent literatures \cite{CCP2011,CS2018,CP2014,CGGS2013,CRSS2018,PCSS2015},  those parameters are modeled as function of the mean firing rate $N(t)$ to incorporate the effect of the firing event to the density function at the macroscopic level. In the simplest form, the follow choice has been widely considered
\begin{equation} \label{eq:parameter}
h(v,N(t))=-v+bN(t), \quad a(N(t))=a_0+a_1N(t).
\end{equation}
In particular, $-v$ describes the leaky behavior, \(b\) models the connectivity of the network: \(b>0\) describes excitatory networks and \(b<0\) inhibitory networks. The connectivity of network plays an important role for the properties of Equation (\ref{eq:fk}), such as steady states and blow-up.
In this paper, we aim to explore reliable and efficient numerical approximations of Equation \eqref{eq:fk} and with parameters given by \eqref{eq:parameter}, which in theory may easily extend to variants of other models.


In the past decade, many researches are devoted to investigating the solution properties of Equation \eqref{eq:fk}, though only limited results are have been obtained due to the nonlinearity and the lack of applicable analysis tools.  In \cite{CP2014,CGGS2013}, the authors study the  existence of classical solutions of Equation (\ref{eq:fk}) and its extensions by linking them to the Stefan problem with a moving boundary and a moving source term. To facilitate designing numerical approximations, we instead consider the following weak version of the solutions:
\begin{definition}
We say a pair of nonnegative functions \((p, N)\) with \\ \(p\in L^\infty( L^1_+( -\infty,V_F)\times \mathbb{R}^+\)), \(N\in L^1_{loc,+}(\mathbb{R}^+)\) is a weak solution of Equation (\ref{eq:fk}) if for any test function \(\varphi(v,t)\in C^\infty((-\infty, v_F]\times[0,T])\) such that \\ \(\partial_{vv}\varphi,v\partial_{v}\varphi\in L^\infty((-\infty, V_F)\times(0, T ))\), we have
\begin{align}
&\int_0^T\int_{-\infty}^{V_F}p(v,t)\left( -a\partial_{vv}\varphi-\partial_t\varphi(v,t)
-h(v,N(t))\partial_v\varphi(v,t) \right) dvdt\nonumber\\
=&\int_{-\infty}^{V_F}\left(p(v,0)\varphi(v,0)-p(v,T)\varphi(v,T)\right) dv+\int_0^TN(t)\varphi(V_R,t)dt-\int_0^TN(t)\varphi(V_F,t)dt.
\end{align}
\end{definition}

Note that, the notation of solution above essentially agrees with Definition (2.1) in \cite{CCP2011}. Also note that, the choice of the test function $\varphi \equiv 1$ naturally implies the weak solution conserves the mass of the initial data.

Even in the weak sense, properties of the Fokker-Planck Equation \eqref{eq:fk}   are not thoroughly understood due to the flux shift and the nonlinearity. However, from the known results and numerical experiments (see, e.g. \cite{CCP2011,CS2018,CCT2010,PCSS2015}), people discover that  the model produces  diverse solutions with complicated structures such as multiple steady states, blow-up behaviour and synchronous states.
For inhibitory and excitatory networks when connectivity is small, it is proved that there is a unique steady state. However when connectivity  is big enough, there may be nonexistence or nonuniqueness of the steady states.
Besides, when \(b>0\), finite-time blow-up phenomenon may appear in the weak sense for certain initial conditions.

It is worth noting that, for long time asymptotic behaviour, there is only limited understanding. As far as we know, no direct energy estimate has been derived yet for such systems. For the linear case \(a_1=b=0\) though, the relative entropy estimate can be proved, which implies the exponential convergence towards equilibrium\cite{CCP2011}.


Numerical studies for Equation (\ref{eq:fk}) and other Fokker-Planck type equations arising in neuroscience are widely open due to the absence of conventional analytic properties, though a number of meaningful numerical experiments have been done. In \cite{CCT2010}, the authors propose a numerical scheme combining the WENO-finite differences and the Chang-Cooper method,
the numerical tests are mainly concerned with the blow-up phenomenon and the steady states.
Another type of numerical tests is to simulate synchronization and periodic solutions for a variant model. In \cite{CP2014,CS2018}, the authors study the leaky integrate-and-fire model with transmission delay. It is shown numerically that, transmission delay not only prevents the blow-up phenomenon, but also may produce periodic solutions, although only partial analysis results are obtained.

To the best of our knowledge, no rigorous numerical analysis has been done for computational methods of such Fokker-Planck equations with a flux shift, though numerical approximations of  general Fokker-Planck type equations have been extensively investigated.
In \cite{CCH2015}, the authors propose an explicit positivity-preserving and entropy-decreasing finite volume scheme for nonlinear nonlocal equations with a gradient flow structure. Another track of numerical methods are based on the symmetrization of the Fokker-Planck equation\cite{JY2011,JW2011,ABPP2018,LWZ2018}, which is also referred to as the Scharfetter-Gummel flux approach  \cite{ABPP2018}, and with proper time discretization, the resulting schemes are often semi-implicit or fully implicit. We emphasize that the numerical methods above do not, however, give insights in treating the flux shift.

In fact, the Equation (\ref{eq:fk}) can also be viewed as a balance law equation
\begin{equation} \label{BL}
\partial_tp+\partial_vF=0,
\end{equation}
with flux shift from \(V_F\) to \(V_R\). For $v\in(-\infty,V_F)/\{V_R\}$, the flux function of the equation is given by
\begin{equation}\label{eq:flux}
 F(v,t)=-a \partial_vp+(-v+bN)p.
\end{equation}
The boundary condition at $v=V_F$ and the derivative jump condition at $v=V_R$ can thus be cast as
\begin{equation}\label{fluxshift}
\begin{cases}
F(V_R^+,t)- F(V_R^-,t)=N(t),\\
F(V_F^-,t)=N(t).
\end{cases}
\end{equation}
In other words, the flux flowing out from \(V_F\) is repositioned at the point \(V_R\). Equivalently, we can consider the modifed  flux function
\begin{equation}\label{eq:mflux}
\tilde F(v,t)=-a \partial_vp+(-v+bN)p - H(v-V_R) N,
\end{equation}
which is continuous on $(-\infty,V_F)$. Here \(H(v)\) stands for the Heaviside function. In this paper, we {present} the numerical scheme based on the balance-law form.
%

As we have mentioned, the solution properties of Equation \eqref{eq:fk} are poorly understood, and thus there is no obvious way to design a stable scheme or an energy dissipating scheme since there is no such relevant results in the continuous case.
In fact, it is known that in the linear equation \(a_1=b=0\), the relative entropy given by
\begin{equation} \label{eq:entropy}
I=\int_{-\infty}^{V_F}G \left( \frac{p(v,t)}{p^{\infty}(v)} \right) p_{\infty}(v)dv,
\end{equation}
can be shown to be decreasing in time (see Theorem 4.2 in \cite{CCP2011}), where \(G\) is a convex function and \(p^{\infty}(v)\) stands for the stationary solution. 
In particular, the dissipation in the relative entropy consists of two parts: the bulk dissipation similar to conventional Fokker-Planck equations and the dissipation due to the jump between $V_R$ and $V_F$. (For completeness, we give the full statement of this result in Section 3.1.)

In this paper, we study the central difference approximation of the Fokker-Planck Equation (\ref{eq:fk}), which is based on the Scharfetter-Gummel reformulation. In the semi-discrete scenario, one may define a discrete relative entropy with a similar form to Equation (\ref{eq:entropy}):
\begin{equation}
S(t)=\sum_{i}hG\left(\frac{p_i}{p_i^{\infty}}\right)p_i^{\infty},
\end{equation}
where $i$ is the spatial index and $h$ is the spatial size, and the meaning of other quantities shall be specified later. In this paper, we prove that \(\frac{d}{dt}S\leq0\) for the linear equation when \(a_1=b=0\), and the discrete dissipation also breaks into the bulk part and the boundary part.
With proper time discretization, the fully discrete scheme is only linearly implicit even when the equation is nonlinear, and hence the use of a nonlinear solver is avoided. Furthermore, we prove that the fully discrete scheme is conservative and conditionally positivity-preserving, which makes it ideal for simulations. With extensive numerical tests, we verify the claimed properties of the methods and demonstrate their superiority in various challenging applications. To our knowledge, the numerical method presented in this paper is the first numerical solver for the Fokker-Planck equation with a flux shift, for which rigorous numerical analysis is provided.

The rest of the paper is outlined as follows. In Section 2 we give a detailed description of the  scheme and prove its basic properties. In Section 3, the discrete relative entropy for the linear model is proved to be decreasing in time at the semi-discrete level. In Section 4 we numerically verify the properties of the proposed scheme and present various numerical experiments. Finally, we give some concluding remarks in Section 5.

\section{Scheme description and numerical analysis}

{In this section we introduce the numerical scheme for Equation \eqref{eq:fk} based on the Scharfetter-Gummel reformulation \cite{JY2011,JW2011,ABPP2018,LWZ2018}, a typical  symmetrization form for Fokker-Planck equations. The constructed numerical scheme preserves various structures of Equation \eqref{eq:fk}, such as the conservation of mass and the preservation of positivity. In particular, we can show that the proposed scheme satisfies  the discrete relative entropy  (see Theorem \ref{thm:nre}), matching the only available long time asymptotic property of the solution to Equation \eqref{eq:fk}.}

\subsection{Description of the scheme}

Now we describe the numerical scheme for the Fokker-Planck Equation  (\ref{eq:fk}) of the neuron networks, where both the semi-discrete scheme and the fully discrete scheme are presented with detailed construction.

\begin{figure}
\centering
\includegraphics[width=14cm,height=4cm]{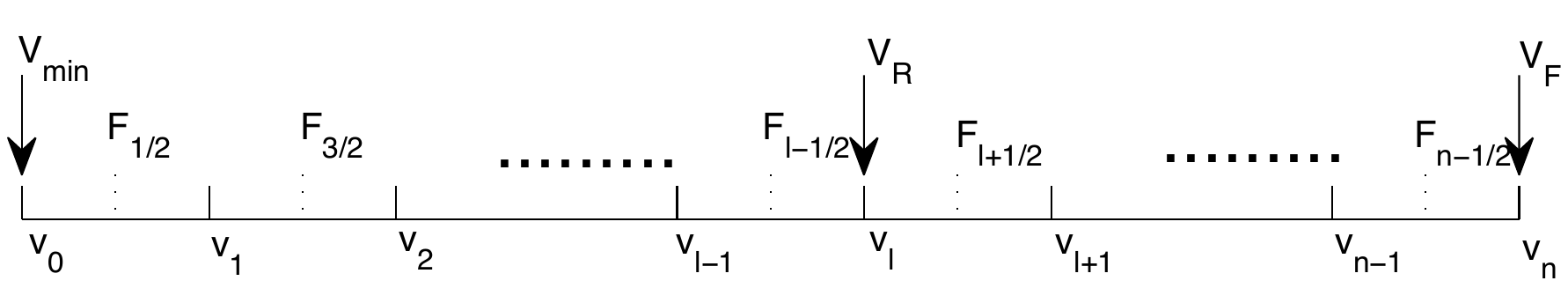}
\caption{Illustration of the scheme stencil.}
\label{figsch}
\end{figure}

As illustrated by Figure \ref{figsch}, we choose a finite interval \([V_{\min},V_F]\) as the computation domain. We suppose $V_{\min}$ is small enough such that the density function is practically negligible near $v=V_{\min}$.  Then we divide the interval into \(n\) equal subintervals with the spatial size \(h=\frac{ V_F-V_{\min}}{n}\). The grid points are chosen as \(v_0=V_{\min},v_1=V_{\min}+h,\cdots,v_n=V_F\). In particular, the reset point \(V_R\) is chosen as a grid point and denoted as \(v_l=V_R\). {Finally, \(p_i(t)\) stands for the semi-discrete solution at $v=v_i$ and \(p_i^m\) stands for the numerical solution at $v=v_i$ and $t=t_m=m\tau$, where $\tau$ is the step size.}

Since Equation \eqref{eq:fk} can be viewed as a generalized balance law, when $v_i \ne v_l$, the approximate grid value of the density function is updated through the form (\ref{BL}):
{\begin{equation}\label{eq:BL}
\frac{\partial p_i}{\partial t}+\frac{F_{i+\frac{1}{2}}-F_{i-\frac{1}{2}}}{h}
=0,
\end{equation}
where the numerical flux \(F_{i+\frac{1}{2}}\) is an approximation to the flux function in \eqref{eq:flux} at the half grid points.} However, due to the derivative  jump of \(p(v,t)\) at \(v=V_R\), we need to incorporate a flux difference of \(N(t)\) from the left hand side and right hand side of \(v_l=V_R\), illustrated in Equation \eqref{fluxshift}. To this end, we introduce a modification to the numerical flux by subtracting a Heaviside function multiplied by the numerical approximation of the mean firing rate, denoted by $N_h$, i.e.
\begin{equation}\label{mod}
\tilde{F}_{i+\frac{1}{2}}=F_{i+\frac{1}{2}}-N_hH\left(v_{i+\frac{1}{2}}-V_R\right).
\end{equation}
Here, we choose to use the first order finite difference approximation for the $N_h$, given by
\begin{equation}\label{eq:Nh}
N_h=-a(N_h)\frac{0-p_{n-1}}{h}.
\end{equation}
We conclude from the expression above that if $p_{n-1}$ is nonnegative, $N_h$ is also nonnegative. We shall see in later sections that this property is consistent with the rest of the numerical scheme such that the overall scheme is positivity preserving conditionally.
On the other hand, naive high order approximation of the mean firing rate may cause \(N_h<0\) for some particular initial data, which results in the instability of the whole scheme. We may explore acceptable high order approximations of the mean firing rate in the future.

With the modified flux, the semi-discrete scheme reads:
\begin{equation}
\label{sch:modflux}
\frac{\partial p_i}{\partial t}+\frac{\tilde{F}_{i+\frac{1}{2}}-\tilde{F}_{i-\frac{1}{2}}}{h}
=0.
\end{equation}
Finally the boundary condition is set as \(p_0=p_n=0\).


Now we introduce the Scharfetter-Gummel reformulation. Let
\[
U(v,t)=:-\frac{\int h(v,N(t))dv}{a(N(t))}=\frac{(v-bN(t))^2}{2a(N(t))},
\]
and thus we combine the diffusion term and the convection term together
\[
\partial_v\left(e^{-U}\partial_v\left(\frac{p}{e^{-U}}\right)\right)=-\frac{\partial_v((-v+bN(t))p(v,t))}{a(N(t))}+\partial_{vv}p(v,t).
\]

For simplicity of notation, we further denote
\begin{equation}
\label{eq:M}
M(v,t)=e^{-U(v,t)}=\exp\left(-\frac{(v-bN(t))^2}{2a(N(t))}\right).
\end{equation}
Then when $v\in(-\infty,V_F)/\{V_R\}$, the Fokker-Planck Equation \eqref{eq:fk} can be written as
\begin{equation} \label{eq:SG1}
\partial_t p(v,t)-a(N(t))\partial_v\left(M(v,t)\partial_v\left(\frac{p(v,t)}{M(v,t)}\right)\right)=0.
\end{equation}

Now we apply the center difference discretization for Equation \eqref{eq:SG1}, the modified numerical flux can be written as (note that the numerical flux of boundary cell will be defined later)
\begin{equation} \label{SGd}
\tilde{F}_{i+\frac{1}{2}}=-a(N_h)M_{i+\frac{1}{2}}\frac{\frac{p_{i+1}}{M_{i+1}}-\frac{p_i}{M_i}}{h}-N_hH\left(v_{i+\frac{1}{2}}-V_R\right),\quad 1\leq i\leq n-2,
\end{equation}
where \(M_i(t)\) and \(M_{i+\frac{1}{2}}(t)\) denote approximations to \(M\) at $v_i$ and $v_{i+\frac{1}{2}}$ as follows:
\begin{equation}\label{eq:Mi}
M_i(t)=\exp\left(-\frac{\left(v_i-bN_h\right)^2}{2a\left(N_h\right)}\right).
\end{equation}



{\begin{remark}
Although the analytical expression of \(M\) and $M_i$ is given in Equation \eqref{eq:M} and \eqref{eq:Mi}, we choose to use \(M_{i+\frac 1 2}^H\), the harmonic mean of $M_i$ and $M_{i+1}$, to approximate $M_{i+\frac 1 2}$, i.e.
\begin{equation} \label{harmonic}
M_{i+ \frac 1 2}^H= \left( \frac 1 2 \Bigl( (M_i)^{-1}+ (M_{i+1})^{-1} \Bigr)   \right)^{-1}.
\end{equation}
We shall show in Section \ref{sec:rel} that this choice is necessary for the discrete relative entropy estimate. However, this choice is not essential for numerical tests in  Section 4.
\end{remark}}

Due to boundary condition \eqref{fluxshift} and in order to make the scheme conservative, numerical fluxes at middle of the boundary cells are defined specially as follows
\[
\begin{cases}
\tilde{F}_{\frac{1}{2}}=F_{\frac{1}{2}}=0,\\
\tilde{F}_{n-\frac{1}{2}}=F_{n-\frac{1}{2}}-N_h=0.
\end{cases}
\]
Notice that, the numerical boundary conditions does not bring in additional difficulty since $N_h$ is treated explicitly in Equation \eqref{eq:Nh}.

In a similar way, we also give the fully discrete scheme according to the balance law form \eqref{BL} as follows:
\begin{equation}\label{eq:BLfull}
\frac{p_i^{m+1}-p_i^m}{\tau}+\frac{\tilde{F}_{i+\frac{1}{2}}-\tilde{F}_{i-\frac{1}{2}}}{h}
=0,
\end{equation}
Then we consider two schemes with different time discretization strategies, both based on the form (\ref{eq:BLfull}). In the first scheme, the numerical flux is fully explicit:
\begin{equation}\label{sch:exp}
\tilde{F}_{i+\frac{1}{2}}^m=-a(N_h^m)\frac{M_{i+\frac{1}{2}}^m}{h}\left(\frac{p_{i+1}^m}{M_{i+1}^m}-\frac{p_i^m}{M_i^m}\right)-N_h^mH\left(v_{i+\frac{1}{2}}-V_R\right),\quad 1\leq i\leq n-2.
\end{equation}
The explicit method obviously suffers from the parabolic type stability constraint, which prevents efficient numerical simulations. The introduction of this method is mainly for the numerical comparison.

Next, we consider the following semi-implicit scheme, where the numerical flux is given by
\begin{eqnarray}\label{sch:semi}
\tilde{F}_{i+\frac{1}{2}}^m=-a(N_h^m)\frac{M_{i+\frac{1}{2}}^{m}}{h}\left(\frac{p_{i+1}^{m+1}}{M_{i+1}^{m}}-\frac{p_i^{m+1}}{M_i^{m}}\right)-N_h^mH\left(v_{i+\frac{1}{2}}-V_R\right),\quad 1\leq i\leq n-2.
\end{eqnarray}
Note that, in the numerical flux \eqref{sch:semi} \(N_h\) and \(M\) are treated explicitly (recall that \(M^m\) is defined through \(N_h^m\) due to Equation \eqref{eq:M}) but the dependence on \(p_i^{m+1}\) is implicit, and thus we avoid solving a nonlinear equation.

\subsection{Numerical analysis}

{In this part, we analyze the proposed schemes, and show some properties with respect to the preservation of the solution structure of the problem. We justify the positivity preserving property for the semi-implicit scheme \eqref{sch:semi}, and show that the formal energy estimation is valid for semi-discrete scheme \eqref{sch:modflux} given the boundedness of the mean firing rate. We also comment on the difficulties of the numerical study of such models, and motivate the significance of the discrete relative entropy estimate to be established in the next section.}

\subsubsection{Positivity preserving}

{The semi-implicit scheme with the numerical flux \eqref{sch:semi} preserves the positivity for all \(p_i^m\)  as stated in the following theorem. }

{\begin{theorem}
Consider the Fokker-Planck Equation \eqref{eq:fk} with the initial condition \(p_0(v)>0\), the semi-implicit scheme \eqref{eq:BLfull} \eqref{sch:semi} preserves positivity, i.e. \(p_i^m>0\), \(\forall i,m\), provided that the following parabolic type constraint is satisfied:
\begin{equation}\label{eq:condpost}
\frac{\tau}{h^2}<\frac{1}{a\left(N_h^m\right)}.
\end{equation}
\end{theorem}}

\begin{proof}
{Let \(\lambda=\frac{\tau}{h}\), the discrete semi-implicit scheme (\ref{eq:BLfull}) with flux \eqref{sch:semi} can be written as follows: for \(1<i<n-1\),
\begin{eqnarray}
\label{positive}
p_i^{m+1}+\lambda\frac{M_{i-\frac{1}{2}}^m}{h}\left(\frac{p_i^{m+1}}{M_i^m}-\frac{p_{i-1}^m}{M_{i-1}^m}\right)
+\lambda\frac{M_{i+\frac{1}{2}}^m}{h}\left(\frac{p_i^{m+1}}{M_i^m}
-\frac{p_{i+1}^{m+1}}{M_{i+1}^m}\right)\nonumber\\
=\lambda\left(H\left(v_{i+\frac{1}{2}}-V_R\right)-H\left(v_{i-\frac{1}{2}}-V_R\right)\right) N^m_h+p_i^m.
\end{eqnarray}
For \(i=1\),
\begin{eqnarray}
\label{positive1}
p_i^{m+1}
+\lambda\frac{M_{i+\frac{1}{2}}^m}{h}\left(\frac{p_i^{m+1}}{M_i^m}
-\frac{p_{i+1}^{m+1}}{M_{i+1}^m}\right)=p_i^m.
\end{eqnarray}
For \(i=n-1\)
\begin{eqnarray}
\label{positiven-1}
p_i^{m+1}+\lambda\frac{M_{i-\frac{1}{2}}^m}{h}\left(\frac{p_i^{m+1}}{M_i^m}-\frac{p_{i-1}^m}{M_{i-1}^m}\right)
=p_i^m-\lambda N_h^m.
\end{eqnarray}}

{The positivity preservation property can be proved by contradiction. If \(p_i^m>0\) for any \(0<i<n\), but \(p_i^{m+1}\leq0\) for some \(0<i<n\). Then we assume that $\frac{p_i^{m+1}}{M_i^m}$ takes the nonpositive minimum at $i=j$. Take \(i=j\) in Equation \eqref{positive}, \eqref{positive1} and \eqref{positiven-1}, we find that the left side of the equations is nonpositive while the right side is positive. (Recall that \(N_h^m\) is given by Equation \eqref{eq:Nh}, then right side of Equation \eqref{positiven-1} is positive due to Equation \eqref{eq:condpost}.) Thus we conclude that the semi-implicit scheme preserves positivity.}
\end{proof}

{\begin{remark}
We remark that it has been shown that the semi-implicit numerical methods based on the Scharfetter-Gummel flux for other related models are unconditionally positivity preserving, see e.g. \cite{JW2011,LWZ2018}. In this light, the proof above indicates the numerical treatment of the flux shift doesn't break the structure provided that Equation \eqref{eq:condpost} is valid, but this constraint manifests the additional challenge in the numerical approximation of such Fokker-Planck equations. However, we emphasize that Equation \eqref{eq:condpost} doesn't seem to be a sharp constraint for the positivity preserving property of the numerical tests as shown in Section 4. In fact, as shown in Section 4.1, when the parabolic type constraint is violated, the full explicit scheme becomes unstable, while the semi-implicit scheme still produces reliable numerical results with the correct convergence order. We also note that The condition \eqref{eq:condpost} can be removed by treating the mean firing rate implicitly, which, unfortunately, results in the need of using nonlinear solvers.
\end{remark}}

\subsubsection{Formal discrete energy estimation}

We would like to draw readers' attention to the fact that there is no energy estimate to the Fokker-Planck Equation \eqref{eq:fk} due to the flux shift. Whereas, it is shown in \cite{CGGS2013} that, when the mean firing rate $N(t)$ is bounded, the classical solution to the Fokker-Planck equation exists. Motivated by this result, we carry out the following formal energy estimate for the semi-discrete scheme in the linear case.

Define the semi-discrete energy as
\begin{equation}\label{energy}
E(t)=\sum_{i=1}^{n-1}p_i\ln\left(\frac{p_i}{M_i}\right)h-C\left(G_l(t)-G_{n-1}(t)\right),
\end{equation}
where
\[
G_i(t)=\int_0^t \ln\left(\frac{p_i\left(s\right)}{M_i}\right)ds.
\]

To show the following estimate, we  further make the following technical assumption: in the semi-discrete scheme \eqref{sch:modflux}, we have
\begin{equation} \label{eq:assum1}
\frac{p_l}{M_l} \ge \frac{p_{n-1}}{M_{n-1}}.
\end{equation}
Recall that $p_l$ is the discrete density at $v=V_R$ and $p_{n-1}$ is the discrete density next to $v=V_F$. From this perspective, this assumption is reasonable due to the boundary condition at $v=V_F$ and the flux shift condition from $V_F$ to $V_R$.

\begin{theorem}
Consider Fokker-Planck Equation \eqref{eq:fk} with \(a \equiv1\) and \(b=0\). { Assume that for a given time $T>0$, there exists a constant $C$ such that for $i=1,\cdots,n-1$,
\begin{equation} \label{eq:assum2}
p_i(t)>0,\quad N_h(t) \le C, \quad \forall t \in [0,T],
\end{equation}
and the technical assumption \eqref{eq:assum1} is valid.} If we apply the semi-discrete scheme \eqref{sch:modflux}, then the semi-discrete energy \eqref{energy} is nonincreasing for $t\in[0,T]$. Moreover, for $\forall t \in [0,T]$, we have
\begin{eqnarray}
\frac{d}{dt}E(t)&=&-\sum_{i=1}^{n-2}\left(\frac{M_{i+\frac{1}{2}}}{h}\left(\frac{p_{i+1}}{M_{i+1}}-
\frac{p_i}{M_i}\right)\left(\ln\left(\frac{p_{i+1}}{M_{i+1}}\right)-\ln\left(\frac{p_i}{M_i}\right)\right)\right)\nonumber\\&&
+(N_h-C)\left(\ln\left(\frac{p_l}{M_l}\right)-\ln\left(\frac{p_{n-1}}{M_{n-1}}\right)\right) \le 0.
\end{eqnarray}
\end{theorem}

\begin{proof}

By direct calculation, we have
\begin{eqnarray}\label{prf:energyest}
\frac{d}{dt}E(t)
&=&\sum_{i=1}^{n-1}\left(\frac{dp_i}{dt}\left(\ln\left(\frac{p_i}{M_i}\right)+1\right)h\right)
-C\left(\ln\left(\frac{p_l}{M_l}\right)-\ln\left(\frac{p_{n-1}}{M_{n-1}}\right)\right)\nonumber\\
&=&\sum_{i=1}^{n-1}\left(\left(\tilde{F}_{i-\frac{1}{2}}-\tilde{F}_{i+\frac{1}{2}}\right)\left(\ln\left(\frac{p_i}{M_i}\right)+1\right)\right)
-C\left(\ln\left(\frac{p_l}{M_l}\right)-\ln\left(\frac{p_{n-1}}{M_{n-1}}\right)\right).\nonumber\\
\end{eqnarray}

Then we  substitute Equation \eqref{SGd} into Equation \eqref{prf:energyest}:
\begin{eqnarray}
\frac{d}{dt}E(t)
&=&\sum_{i=1}^{n-1}\left(\left(F_{i-\frac{1}{2}}-F_{i+\frac{1}{2}}\right)\left(\ln\left(\frac{p_i}{M_i}\right)+1\right)\right)+N_h\left(\ln\left(\frac{p_l}{M_l}\right)+1\right)\nonumber\\&&
-C\left(\ln\left(\frac{p_l}{M_l}\right)-\ln\left(\frac{p_{n-1}}{M_{n-1}}\right)\right)\nonumber\\
&=&\sum_{i=1}^{n-2}\left(F_{i+\frac{1}{2}}\left(\ln\left(\frac{p_{i+1}}{M_{i+1}}\right)-\ln\left(\frac{p_i}{M_i}\right)\right)\right)
+N_h\left(\ln\left(\frac{p_l}{M_l}\right)-\ln\left(\frac{p_{n-1}}{M_{n-1}}\right)\right)\nonumber\\&&
-C\left(\ln\left(\frac{p_l}{M_l}\right)-\ln\left(\frac{p_{n-1}}{M_{n-1}}\right)\right)\nonumber\\
&=&-\sum_{i=1}^{n-2}\left(\frac{M_{i+\frac{1}{2}}}{h}\left(\frac{p_{i+1}}{M_{i+1}}-
\frac{p_i}{M_i}\right)\left(\ln\left(\frac{p_{i+1}}{M_{i+1}}\right)-\ln\left(\frac{p_i}{M_i}\right)\right)\right)\nonumber\\&&
+(N_h-C)\left(\ln\left(\frac{p_l}{M_l}\right)-\ln\left(\frac{p_{n-1}}{M_{n-1}}\right)\right).
\end{eqnarray}

Due to the monotonicity of the $\ln(x) $ function, we have
\[
-\sum_{i=1}^{n-2}\left(\frac{M_{i+\frac{1}{2}}}{h}\left(\frac{p_{i+1}}{M_{i+1}}-
\frac{p_i}{M_i}\right)\left(\ln\left(\frac{p_{i+1}}{M_{i+1}}\right)-\ln\left(\frac{p_i}{M_i}\right)\right)\right) \le 0,
\]
which corresponds to the bulk energy dissipation.

And due to the assumptions \eqref{eq:assum1} and \eqref{eq:assum2}, we also conclude
\[
(N_h-C)\left(\ln\left(\frac{p_l}{M_l}\right)-\ln\left(\frac{p_{n-1}}{M_{n-1}}\right)\right) \le 0,
\]
which can be interpreted as the energy dissipation due to the flux shift.

%
\end{proof}

We remark that  there is no obvious way to know a priori whether the assumptions \eqref{eq:assum1} and \eqref{eq:assum2} are satisfied. And we observe in the proof above that even in the linear case, it is not clear whether and how the flux shift condition introduces  dissipation to the system. These issues make the value of the formal discrete energy estimate rather limited.
In fact, in the nonlinear case, the solution to the Fokker-Planck equation may blowup when clearly either \eqref{eq:assum1} or \eqref{eq:assum2} is violated in finite time.

In the next section, we revisit the long time behavior of the numerical scheme from the perspective of the relative entropy, which proves to be a more suitable metric.
%

\section{Discrete relative entropy estimate} \label{sec:rel}

Since the solution properties of Equation \eqref{eq:fk} are only partially understood, there is no obvious way to design a scheme with valid long time asymptotic behavior. In this section, we consider the stability of the scheme we have proposed in the view of relative entropy. We briefly review the steady states and relative entropy results first (see \cite{CCP2011} for a full discussion), then prove that our numerical scheme is associated with a discrete relative entropy, which is nonincreasing in time as well.

\subsection{Steady states and relative entropy for the continuous problem}

First we give the definition of stationary solution of Equation \eqref{eq:fk}. We denote by \(p^{\infty}(v)\)  the density function of the stationary state, which satisfies
\begin{equation}\label{eq:ss}
\begin{cases}
\partial_v (h(v,N^\infty)p^\infty) - a(N^\infty)\partial_{vv}p^\infty
=0, \quad v\in(-\infty,V_F)/\{V_R\},  \\
p^\infty(V_R^-)=p^\infty(V_R^+),\qquad \frac{\partial}{\partial v}p^\infty(V_R^-)=\frac{\partial}{\partial v} p^\infty(V_R^+)+\frac{N^\infty}{a(N^\infty)}.
\end{cases}
\end{equation}
Here \(N^\infty\) indicates the firing rate for stationary solution:
\begin{equation}\label{NinftyC}
N^{\infty}=-a(N^{\infty})\partial_v p^{\infty}(V_F).
\end{equation}

Given the firing rate \(N^{\infty}\), the expression of \(p^{\infty}(v)\) is given by
\begin{equation}\label{stationary}
p^{\infty}(v)=\frac{N^{\infty}}{a(N^{\infty})}e^{-\frac{h\left(v,N^{\infty}\right)^2}{2a\left(N^{\infty}\right)}}
\int_{\max\{v,V_r\}}^{V_f}e^{\frac{h\left(\omega,N^{\infty}\right)^2}{2a\left(N^{\infty}\right)}}d\omega.
\end{equation}
In \cite{CCP2011,PCSS2015}, it is shown that for inhibitory networks (\(b\leq0\)) and excitatory networks when the connectivity \(b\) is small, there is a unique steady state (the linear case when \(b=0\) is an example with unique stationary solution). However, when the connectivity \(b\) is big enough, nonexistence or nonuniqueness may happen, which depends on the initial condition and the parameters in the equation.

Then we define the relative entropy function for Equation \eqref{eq:fk}:
\begin{equation}\label{re}
I=\int_{-\infty}^{V_F}G\left(\frac{p(v,t)}{p^{\infty}(v)}\right)p_{\infty}(v)dv.
\end{equation}
The main theorem for relative entropy is first introduced in Theorem 4.2 of \cite{CCP2011}:
\begin{theorem}
\label{thm:re}
Consider Fokker-Planck Equation \eqref{eq:fk} with \(a(N(t))\equiv1\) and \(b=0\). Assume that \(p^{\infty}(v)\) is given by Equation \eqref{eq:ss} and the relative entropy is given by \eqref{re}, where \(G(x)\) is a convex function. Then we have
\begin{eqnarray}\label{eq:red}
\frac{d}{dt}I(t)&=&-\int_{-\infty}^{V_F}p^{\infty}G''\left(\frac{p}{p^{\infty}}\right)\left(\partial_v\left(\frac{p}{p^{\infty}}\right)\right)^2dv
\nonumber\\&&-N^{\infty}\left(G\left(\frac{N}{N^{\infty}}\right)-G\left(\frac{p}{p^{\infty}}\right)-\left(\frac{N}{N^{\infty}}-\frac{p}{p^{\infty}}\right)G'\left(\frac{p}{p^{\infty}}\right)\right)\Bigg|_{V_R}\leq 0.\nonumber\\
\end{eqnarray}
\end{theorem}


The right hand side of Equation \eqref{eq:red} can be divided into two parts. The first part
\[
-\int_{-\infty}^{V_F}p^{\infty} G''\left(\frac{p}{p^{\infty}}\right)\left(\partial_v\left(\frac{p}{p^{\infty}}\right)\right)^2dv
\]
is nonpositive due to the convexity of \(G\). This is a familiar bulk dissipation term.

The remaining part
\[
-N^{\infty}\left(G\left(\frac{N}{N^{\infty}}\right)-G\left(\frac{p}{p^{\infty}}\right)-\left(\frac{N}{N^{\infty}}-\frac{p}{p^{\infty}}\right)
G'\left(\frac{p}{p^{\infty}}\right)\right)\Bigg|_{V_R}
\]
consists of several boundary terms caused by flux jump. Specifically, sum of boundary terms are nonpositive when putting together due to the convexity of $G$. 

We highlight that in the proof of Theorem \ref{thm:re} (Theorem 4.2 of \cite{CCP2011}), the authors put forward the following propositions that play an important role in the proof. We list them here to be compared with their discrete counterparts.
\begin{proposition}
\label{Lemma1}
\begin{equation}
p^{\infty}\partial_v\left(\frac{p(v,t)}{p^{\infty}(v)}\right)=\partial_vp-\frac{p(v,t)}{p^{\infty}(v)}\partial_vp^{\infty}.
\end{equation}
\end{proposition}
\begin{proposition}
\label{Lemma2}
\begin{equation}\label{intro:ss}
F^\infty(v)=-vp^{\infty}-\partial_vp^{\infty}=N^{\infty}H(v-V_R),
\end{equation}
where \(F^\infty(v)\) stands for the flux function for \eqref{eq:ss} and \(H(v)\) denotes the Heaviside function.
\end{proposition}


The first proposition is an important equality that gives the relation between \(p(v,t)\) and \(p^\infty(t)\). The second proposition is a distinct deduction of Equation \eqref{eq:ss}, describing the property of stationary solution. In the proof for discrete relative entropy, the discrete versions of the two propositions are necessarily useful.

\subsection{Numerical relative entropy}

In this subsection, we aim to show the discrete relative entropy estimate, which is a key feature of the proposed semi-discrete scheme. We assume \(a(N(t))\equiv1\) and \(b=0\) in this subsection, which matches the known result of continuous counterpart\cite{CCP2011}.

Note that the flux function of Equation \eqref{eq:fk} (see Equation \eqref{eq:flux}) has two terms but the numerical flux given by Scharfetter-Gummel discretization (see Equation \eqref{SGd}) only has one term. So it is nature to divide the Scharfetter-Gummel flux into two terms like Equation \eqref{eq:flux}:
\begin{equation}\label{divide}
F_{i+\frac{1}{2}}(t)=-M_{i+\frac{1}{2}}\frac{\frac{p_{i+1}}{M_{i+1}}-\frac{p_i}{M_i}}{h}=-\frac{p_{i+1}-p_i}{h}-g_{i+\frac{1}{2}}\frac{p_i+p_{i+1}}{2},
\end{equation}
where
\begin{equation}\label{eq:g}
{g_{i+\frac{1}{2}}=\frac{2}{h}\frac{M_i-M_{i+1}}{M_i+M_{i+1}}\approx v_{i+\frac{1}{2}}.}
\end{equation}
Equation \eqref{divide} holds if we take \(M_{i+\frac{1}{2}}\) as harmonic mean (\ref{harmonic}) of \(M_i\) and \(M_{i+1}\), i.e. \(M_{i+\frac{1}{2}}=M_{i+\frac{1}{2}}^H\). {We also note that \(g_{i+\frac{1}{2}}\) is constant in \(t\) since we have assumed that $a\equiv 1$ and \(b=0\) (see equation \eqref{eq:Mi} for definition of \(M_i\)).} Through modification of Equation \eqref{divide}, we can apply the techniques that are used in the continuous case.

{To introduce the numerical relative entropy, it is necessary to define \(p^\infty_i\) as an approximation to the stationary solution at the grid points.  We assume the Dirichlet boundary condition \(p_0^\infty=p_n^\infty=0\) on the numerical stationary solution. For \(p_i^\infty\)(\(1\leq i\leq n-1\)), we aim to determine them through \(n-1\) independent equations.}

{Firstly, we look at the stationary firing rate. We assume that the discrete stationary firing rate \(N_h^\infty\) equals the continuous stationary firing rate \(N^\infty\) given in Equation \eqref{NinftyC}, i.e, \(N_h^\infty=N^\infty\). We emphasize that \(N_h^\infty\) is considered as a fixed parameter when determining \(p_i^\infty\).}

{Like Equation \eqref{eq:Nh}, we also apply first order finite difference approximation for the numerical stationary firing rate \(N^\infty_h\) and we put forward an equation for \(p_{n-1}^\infty\):
\begin{equation}
\label{disfr}
-\frac{0-p_{n-1}^{\infty}}{h}=N^\infty
\end{equation}
Then we can derive the 'discrete L'Hospital rule':
\begin{equation}\label{disLh}
\frac{N_h}{N^{\infty}_h}=\frac{p_{n-1}}{p_{n-1}^{\infty}}.
\end{equation}}

{In light of Equation \eqref{intro:ss} and \eqref{divide}, it is also reasonable to impose that \(p^\infty_i\)(\(1\leq i\leq n-1\)) satisfies
\begin{equation}\label{nss}
F_{i+\frac{1}{2}}^{\infty}=
-\frac{p_{i+1}^{\infty}-p_i^{\infty}}{h}-g_{i+\frac{1}{2}}\frac{p_i^{\infty}+p_{i+1}^{\infty}}{2}
=N^{\infty}_hH\left(v_{i+\frac{1}{2}}-V_R\right),\quad 1\leq i\leq n-2,
\end{equation}
which is a discrete form of Equation \eqref{divide}.} 

{Like the continuous stationary solution defined in Equation \eqref{stationary} with given \(N^\infty\), \(p_i^\infty\)(\(1\leq i\leq n-1\)) can be solved with given \(N_h^\infty\) through Equation \eqref{disfr} and \eqref{nss}. To show this, we take modification on Equation \eqref{nss}:
\begin{equation}\label{pinftyR}
\left(\frac{1}{h}-\frac{g_{i+\frac{1}{2}}}{2}\right)p_{i+1}^\infty+N_h^\infty H\left(v_{i+\frac{1}{2}}-V_R\right)=\left(\frac{1}{h}+\frac{g_{i+\frac{1}{2}}}{2}\right)p_i^\infty,\quad 1\leq i\leq n-2.
\end{equation}
Since \(p_{i-1}^\infty\) can be solved in Equation \eqref{disfr}, \(p_i^\infty\)(\(i=n-2,n-1,\cdots,1\)) can be determined one by one through Equation \eqref{pinftyR}.}

{To finish the proof of Theorem \ref{thm:nre}, we further need to ensure that \(p_i^\infty>0\). Hence we make a technique assumption as follows:
\begin{equation}\label{TechAssump}
\left|g_{i+\frac{1}{2}}\right|\leq\frac{2}{h}.
\end{equation}
Since \(g_{i+\frac{1}{2}}\) is bounded for \(1\leq i\leq n-2\) due to Equation \eqref{eq:g}, Equation \eqref{TechAssump} is valid provided that \(h\) is small enough. According to equation \eqref{disfr}, if we set \(N_h=N^\infty>0\), then the numerical stationary solution \(p_i^\infty\) is positive for every \(i\). }

After preparation above, we can give the definition of numerical relative entropy:
\begin{equation}\label{nre}
S(t)=\sum_{i=1}^{n-1}hG\left(\frac{p_i}{p_i^{\infty}}\right)p_i^{\infty}.
\end{equation}

{\begin{theorem}
\label{thm:nre}
Consider the Fokker-Planck Equation \eqref{eq:fk} when \(a\equiv1\), \(b=0\). Assume that \(p^{\infty}_i\) satisfy Equation \eqref{nss} and \(N_h^\infty\) satisfies Equation \eqref{disfr}, and the technique assumption \eqref{TechAssump} is satisfied. Consider the semi-discrete scheme \eqref{sch:modflux}, which is associated with the discrete relative entropy defined by Equation \eqref{nre}, where \(G(x)=\frac{1}{2}(x-1)^2\). Then for any spatial size \(h\),  the discrete relative entropy is nonincreasing in time, i.e.
\begin{eqnarray}\label{eq:nred}
\frac{d}{dt}S(t)
&=&\sum_{i=1}^{n-2}\left(G'\left(\frac{p_{i+1}}{p_{i+1}^{\infty}}\right)-G'\left(\frac{p_i}{p_i^{\infty}}\right)\right)
\left(\frac{p_{i+1}}{p_{i+1}^{\infty}}-\frac{p_i}{p_i^{\infty}}\right)
\left(\left(-\frac{1}{2h}-\frac{g_{i+\frac{1}{2}}}{4}\right)p_{i+1}^\infty\right.\nonumber\\&&\left.+\left(-\frac{1}{2h}+\frac{g_{i+\frac{1}{2}}}{4}\right)p_{i}^\infty\right)-\frac{N^{\infty}_h}{2}\left(\frac{p_l}{p_l^{\infty}}-\frac{N_h}{N^{\infty}_h}\right)^2\leq0.
\end{eqnarray}
\end{theorem}}

Note that in Theorem \ref{thm:nre}, a specific quadratic choice of convex function \(G(x)=\frac{1}{2}(x-1)^2\) is made since \(G\) is also chose as a quadratic function in continuous relative entropy study in \cite{CCP2011}. In fact, even a special choice of convex function \(G\) is meaningful for the numerical analysis as a first long time asymptotic property for numerical scheme of Equation \eqref{eq:fk}.

Before showing the proof, we {deduce} the relations between \(p_i\) and \(p_i^\infty\). The following equality is the discrete analogue of Proposition 3.2, which serves as a necessary component for the proof of Theorem \ref{thm:nre},
\begin{equation}\label{trans}
\frac{p_{i+1}-p_i}{h}=\frac{p_i^{\infty}+p_{i+1}^{\infty}}{2}
\frac{1}{h}\left(\frac{p_{i+1}}{p_{i+1}^{\infty}}-\frac{p_i}{p_i^{\infty}}\right)
+\frac{1}{2}\left(\frac{p_{i+1}}{p_{i+1}^{\infty}}+\frac{p_i}{p_i^{\infty}}\right)\frac{p_{i+1}^{\infty}-p_i^{\infty}}{h}.
\end{equation}
The proof of Equation \eqref{trans} is rather straightforward and thus is omitted.

\begin{proof}
First we calculate the derivative of numerical relative entropy \eqref{nre}. According to Equation \eqref{mod} and \eqref{sch:modflux}:
\begin{eqnarray}
\label{prf:mod}
\frac{d}{dt}S(t)
&=&\sum_{i=1}^{n-1}G'\left(\frac{p_i}{p_i^{\infty}}\right)\left(\tilde{F}_{i-\frac{1}{2}}-\tilde{F}_{i+\frac{1}{2}}\right)
\nonumber\\
&=&N_hG'\left(\frac{p_l}{p_l^{\infty}}\right)+\sum_{i=1}^{n-1}G'\left(\frac{p_i}{p_i^{\infty}}\right)
\left(F_{i-\frac{1}{2}}-F_{i+\frac{1}{2}}\right)
\end{eqnarray}

Then plug Equation \eqref{divide} into Equation \eqref{prf:mod}:
\begin{eqnarray}\label{eq:sddre}
\frac{d}{dt}S(t)
&=&N_hG'\left(\frac{p_l}{p_l^{\infty}}\right)-N_hG'\left(\frac{p_{n-1}}{p_{n-1}^{\infty}}\right)\nonumber\\&&
+\sum_{i=1}^{n-2}\left(G'\left(\frac{p_{i+1}}{p_{i+1}^{\infty}}\right)-G'\left(\frac{p_i}{p_i^{\infty}}\right)\right)
\left(-\frac{p_{i+1}-p_i}{h}-g_{i+\frac{1}{2}}\frac{p_i+p_{i+1}}{2}\right)\nonumber\\
&=&N_hG'\left(\frac{p_l}{p_l^{\infty}}\right)-N_hG'\left(\frac{p_{n-1}}{p_{n-1}^{\infty}}\right)+S_1+S_2,
\end{eqnarray}
where
\[
S_1=\sum_{i=1}^{n-2}\left(G'\left(\frac{p_{i+1}}{p_{i+1}^{\infty}}\right)-G'\left(\frac{p_i}{p_i^{\infty}}\right)\right)\left(-\frac{p_{i+1}-p_i}{h}\right),
\]
and
\[
S_2=\sum_{i=1}^{n-2}\left(G'\left(\frac{p_{i+1}}{p_{i+1}^{\infty}}\right)-G'\left(\frac{p_i}{p_i^{\infty}}\right)\right)\left(-g_{i+\frac{1}{2}}\frac{p_i+p_{i+1}}{2}\right).
\]

Plug Equation \eqref{trans} into \(S_1\), and we have
\begin{eqnarray}
\label{prf:s1}
S_1
&=&\sum_{i=1}^{n-2}\left(G'\left(\frac{p_{i+1}}{p_{i+1}^{\infty}}\right)-G'\left(\frac{p_i}{p_i^{\infty}}\right)\right)
\left(-\frac{p_i^{\infty}+p_{i+1}^{\infty}}{2}\frac{1}{h}\left(\frac{p_{i+1}}{p_{i+1}^{\infty}}-\frac{p_i}{p_i^{\infty}}\right)\right)\nonumber\\&&
+\sum_{i=1}^{n-2}\left(G'\left(\frac{p_{i+1}}{p_{i+1}^{\infty}}\right)-G'\left(\frac{p_i}{p_i^{\infty}}\right)\right)\left(-\frac{1}{2}\left(\frac{p_{i+1}}{p_{i+1}^{\infty}}+
\frac{p_i}{p_i^{\infty}}\right)\frac{p_{i+1}^{\infty}-p_i^{\infty}}{h}\right).
\end{eqnarray}

Next with Equation \eqref{nss}, we can further rewrite
\begin{eqnarray}\label{eq:s2}
S_1
&=&\sum_{i=1}^{n-2}\left(G'\left(\frac{p_{i+1}}{p_{i+1}^{\infty}}\right)-G'\left(\frac{p_i}{p_i^{\infty}}\right)\right)
\left(-\frac{p_i^{\infty}+p_{i+1}^{\infty}}{2}\frac{1}{h}\left(\frac{p_{i+1}}{p_{i+1}^{\infty}}-\frac{p_i}{p_i^{\infty}}\right)\right)
\qquad(a)\nonumber\\&&
+\sum_{i=1}^{n-2}\left(G'\left(\frac{p_{i+1}}{p_{i+1}^{\infty}}\right)-G'\left(\frac{p_i}{p_i^{\infty}}\right)\right)
\frac{g_{i+\frac{1}{2}}}{2}\left(\frac{p_{i+1}}{p_{i+1}^{\infty}}+
\frac{p_i}{p_i^{\infty}}\right)\frac{p_i^{\infty}+p_{i+1}^{\infty}}{2}\qquad(b)\nonumber\\&&
+N^{\infty}_h\sum_{i=l}^{n-2}\left(G'\left(\frac{p_{i+1}}{p_{i+1}^{\infty}}\right)-G'\left(\frac{p_i}{p_i^{\infty}}\right)\right)
\frac{1}{2}\left(\frac{p_{i+1}}{p_{i+1}^{\infty}}+\frac{p_i}{p_i^{\infty}}\right).\qquad\qquad\qquad(c)
\nonumber\\
\end{eqnarray}

Now we group the right hand side of Equation \eqref{eq:s2} into three parts (marked by letter): Term (a) is the nonpositive term because \(G\) is a convex function. Term (b) is close to \(-S_2\), with which it can be simplified due to cancellation. Term (c) still needs further simplification. So now we concentrate on term (b) and term (c).

We compare term (b) of Equation \eqref{eq:s2} to \(-S_2\), which is denoted by $S_3$ in the following:
{\begin{eqnarray}
S_3&=&\sum_{i=1}^{n-2}\left(G'\left(\frac{p_{i+1}}{p_{i+1}^{\infty}}\right)-G'\left(\frac{p_i}{p_i^{\infty}}\right)\right)
\frac{g_{i+\frac{1}{2}}^{\infty}}{2}\left(\frac{p_{i+1}}{p_{i+1}^{\infty}}+
\frac{p_i}{p_i^{\infty}}\right)\frac{p_i^{\infty}+p_{i+1}^{\infty}}{2}-(-S_2)\nonumber\\
&=&\sum_{i=1}^{n-2}\left(G'\left(\frac{p_{i+1}}{p_{i+1}^{\infty}}\right)-G'\left(\frac{p_i}{p_i^{\infty}}\right)\right)
\left(\frac{g_{i+\frac{1}{2}}}{2}\left(\frac{p_{i+1}}{p_{i+1}^{\infty}}+\frac{p_i}{p_i^{\infty}}\right)\frac{p_i^{\infty}+p_{i+1}^{\infty}}{2}
-g_{i+\frac{1}{2}}\frac{p_i+p_{i+1}}{2}\right)\nonumber\\
&=&-\sum_{i=1}^{n-2}\left(G'\left(\frac{p_{i+1}}{p_{i+1}^{\infty}}\right)-G'\left(\frac{p_i}{p_i^{\infty}}\right)\right)
\frac{g_{i+\frac{1}{2}}}{4}\left(p_{i+1}^{\infty}-p_i^{\infty}\right)\left(\frac{p_{i+1}}{p_{i+1}^{\infty}}-\frac{p_i}{p_i^{\infty}}\right).
\label{term b}
\end{eqnarray}}

%

{Though \(S_3\) may not be nonpositive, we aim to show that \(S_3\) added with the nonpositive term (a) in Equation \eqref{eq:s2} is still nonpositive:
\begin{eqnarray}
S_4&=&\sum_{i=1}^{n-2}\left(G'\left(\frac{p_{i+1}}{p_{i+1}^{\infty}}\right)-G'\left(\frac{p_i}{p_i^{\infty}}\right)\right)
\left(-\frac{p_i^{\infty}+p_{i+1}^{\infty}}{2}\frac{1}{h}\left(\frac{p_{i+1}}{p_{i+1}^{\infty}}-\frac{p_i}{p_i^{\infty}}\right)\right)\nonumber\\
&&-\sum_{i=1}^{n-2}\left(G'\left(\frac{p_{i+1}}{p_{i+1}^{\infty}}\right)-G'\left(\frac{p_i}{p_i^{\infty}}\right)\right)
\frac{g_{i+\frac{1}{2}}}{4}\left(p_{i+1}^{\infty}-p_i^{\infty}\right)\left(\frac{p_{i+1}}{p_{i+1}^{\infty}}-\frac{p_i}{p_i^{\infty}}\right)\nonumber\\
&=&\sum_{i=1}^{n-2}\left(G'\left(\frac{p_{i+1}}{p_{i+1}^{\infty}}\right)-G'\left(\frac{p_i}{p_i^{\infty}}\right)\right)
\left(\frac{p_{i+1}}{p_{i+1}^{\infty}}-\frac{p_i}{p_i^{\infty}}\right)
\left(\left(-\frac{1}{2h}-\frac{g_{i+\frac{1}{2}}}{4}\right)p_{i+1}^\infty\right.\nonumber\\&&\left.+\left(-\frac{1}{2h}+\frac{g_{i+\frac{1}{2}}}{4}\right)p_{i}^\infty\right).
\end{eqnarray}
Recall that we have \(p_i^\infty>0\) provided that the technique assumption \eqref{TechAssump} is valid, then we conclude \(S_4\leq0\).}

Recall that \(G(x)=\frac{1}{2}(x-1)^2\), we can rewrite term (c) of Equation \eqref{eq:s2} as follows:
\begin{eqnarray}\label{termc}
&&N^{\infty}_h\sum_{i=l}^{n-2}\left(G'\left(\frac{p_{i+1}}{p_{i+1}^{\infty}}\right)-G'\left(\frac{p_i}{p_i^{\infty}}\right)\right)
\frac{1}{2}\left(\frac{p_{i+1}}{p_{i+1}^{\infty}}+\frac{p_i}{p_i^{\infty}}\right)\nonumber\\
&=&\frac{N^{\infty}_h}{2}\sum_{i=l}^{n-2}
\left(\left(\frac{p_{i+1}}{p_{i+1}^{\infty}}\right)^2-\left(\frac{p_i}{p_i^{\infty}}\right)^2\right)\nonumber\\
&=&\frac{N^{\infty}_h}{2}\left(\left(\frac{p_{n-1}}{p_{n-1}^{\infty}}\right)^2-\left(\frac{p_{l}}{p_{l}^{\infty}}\right)^2\right).
\end{eqnarray}
Therefore term (c) only consists boundary term.

Plug Equation \eqref{eq:s2}, \eqref{term b} and \eqref{termc} into Equation \eqref{eq:sddre}, the derivative of numerical relative entropy can be simplified:
\begin{eqnarray}\label{final}
\frac{d}{dt}S(t)&=& S_4+N_h\left(\frac{p_l}{p_l^{\infty}}-\frac{p_{n-1}}{p_{n-1}^{\infty}}\right)
+\frac{N^{\infty}_h}{2}\left(\left(\frac{p_{n-1}}{p_{n-1}^{\infty}}\right)^2-\left(\frac{p_{l}}{p_{l}^{\infty}}\right)^2\right).
\end{eqnarray}

With Equation \eqref{disLh} and \eqref{TechAssump}, we finally obtain
\begin{eqnarray}
\label{prf:final}
\frac{d}{dt}S(t)
&=&S_4+N_h\left(\frac{p_l}{p_l^{\infty}}-\frac{N_h}{N_h^{\infty}}\right)
+\frac{N^{\infty}_h}{2}\left(\left(\frac{N_h}{N^{\infty}_h}\right)^2-\left(\frac{p_{l}}{p_{l}^{\infty}}\right)^2\right)\nonumber\\
&=&\sum_{i=1}^{n-2}\left(G'\left(\frac{p_{i+1}}{p_{i+1}^{\infty}}\right)-G'\left(\frac{p_i}{p_i^{\infty}}\right)\right)
\left(\frac{p_{i+1}}{p_{i+1}^{\infty}}-\frac{p_i}{p_i^{\infty}}\right)
\left(\left(-\frac{1}{2h}-\frac{g_{i+\frac{1}{2}}}{4}\right)p_{i+1}^\infty\right.\nonumber\\&&\left.+\left(-\frac{1}{2h}+\frac{g_{i+\frac{1}{2}}}{4}\right)p_{i}^\infty\right)-\frac{N^{\infty}_h}{2}\left(\frac{p_l}{p_l^{\infty}}-\frac{N_h}{N^{\infty}_h}\right)^2\leq0.
\end{eqnarray}

\end{proof}

\begin{remark}
If we compare the discrete relative entropy estimate \eqref{prf:final} with its continuous version \eqref{eq:red}, the term
\begin{equation}
\label{rmk:sumterm}
\sum_{i=1}^{n-2}\left(G'\left(\frac{p_{i+1}}{p_{i+1}^{\infty}}\right)-G'\left(\frac{p_i}{p_i^{\infty}}\right)\right)
\left(\frac{p_{i+1}}{p_{i+1}^{\infty}}-\frac{p_i}{p_i^{\infty}}\right)
\left(\left(-\frac{1}{2h}-\frac{g_{i+\frac{1}{2}}}{4}\right)p_{i+1}^\infty+\left(-\frac{1}{2h}+\frac{g_{i+\frac{1}{2}}}{4}\right)p_{i}^\infty\right)
\end{equation}
is the bulk dissipation term, corresponding to the integration term in Equation \eqref{eq:red}
\[
-\int_{-\infty}^{V_F}p^{\infty} G''\left(\frac{p}{p^{\infty}}\right)\left(\partial_v\left(\frac{p}{p^{\infty}}\right)\right)^2dv.
\]
While
\begin{equation}
\label{rmk:boundterm}
-\frac{N^{\infty}_h}{2}\left(\frac{p_l}{p_l^{\infty}}-\frac{N_h}{N^{\infty}_h}\right)^2
\end{equation}
are the boundary terms that account for the contribution of the flux shift, which are analogous to:
\[
-N^{\infty}\left(G\left(\frac{N}{N^{\infty}}\right)-G\left(\frac{p}{p^{\infty}}\right)-\left(\frac{N}{N^{\infty}}-\frac{p}{p^{\infty}}\right)
G'\left(\frac{p}{p^{\infty}}\right)\right)\Bigg|_{V_R}
\]
in Equation \eqref{eq:red}.

In  previous literatures, see \cite{JW2011} for example, numerical methods that are based on the Scharfetter-Gummel flux approach for  Fokker-Planck equations without  flux shift, the bulk dissipation term is common for numerical relative entropy. Our work shows that the contribution of flux shift is also nonpositive, given by Equation \eqref{rmk:boundterm}.

\end{remark}



\section{Numerical Tests}

In this section we verify the properties of the proposed fully-discrete scheme through a series of numerical tests. We test both the explicit scheme (given by Equation \eqref{sch:exp}) and the semi-implicit scheme (given by Equation \eqref{sch:semi}). In our simulations we choose a uniform mesh in \(v\), for \(v\in[V_{\min}, V_F]\). The value \(V_{\min}\) (less than \(V_R\)) is adjusted in the numerical experiments to fulfill that \(p(V_{\min},t)\approx 0\), while \(V_F\) is fixed to $2$. Without special notice, \(V_R\) is set to be $1$ and \(V_{\min}\) is set to be $-4$.

The tests are structured as follows. In subsection \ref{order} we test the order of accuracy in both space and time. Then in subsection \ref{longtime}, we consider different dynamic behaviors of the solutions, including stationary solutions, blow-up solutions and relative entropies for both inhibitory systems and excitatory systems. Finally in subsection \ref{modify}, we consider a modified NNLIF model involving the transmission delay and the refractory state. We introduce numerical schemes for the modified model according to the Scharfetter-Gummel discretization and show that some numerical solutions exhibit time periodic structures.

We choose two types of distributions as initial conditions. The first one is the Gaussian distribution:
\begin{equation}
p_{G}(v)=\frac{1}{\sqrt{2\pi}\sigma_0 M_0}e^{-\frac{(v-v_0)^2}{2\sigma_0^2}},
\end{equation}
where \(v_0\) and \(\sigma_0\) are two given parameters and \(M_0\) denotes a normalization factor such that
\[
\int_{-\infty}^{V_F}p_{M}(v)dv=1.
\]
The other one is the stationary distribution with the equilibrium firing rate \(N^\infty\):
\begin{equation}
p^{\infty}(v)=\frac{N^{\infty}}{a(N^{\infty})}e^{-\frac{h\left(v,N^{\infty}\right)^2}{2a\left(N^{\infty}\right)}}
\int_{\max\{v,V_r\}}^{V_f}e^{\frac{h\left(\omega,N^{\infty}\right)^2}{2a\left(N^{\infty}\right)}}d\omega,
\label{ICstationary}
\end{equation}
where the firing rate \(N^\infty\) is chosen to satisfy
\[
\int_{-\infty}^{V_F}p^\infty(v)dv=1.
\]
In fact, the stationary distribution is a steady solution to Equation \eqref{eq:fk}.

\subsection{Order of accuracy}
\label{order}

In this part, we test the order of accuracy of the schemes. Since the exact solution is unavailable, we estimate the order of the error by
\[
O_h=\log_2 \frac{||\omega_h-\omega_{\frac{h}{2}}||}{||\omega_{\frac{h}{2}}-\omega_{\frac{h}{4}}||},
\]
where \(\omega_h\) is the numerical solution with step length \(h\). The term \(O_h\) above is an approximation for the accuracy order. Both \(L^1\) norm and \(L^\infty\) norm are considered.

Here we choose the Gaussian initial condition with \(v_0=0\) and \(\sigma_0^2=0.25\) and \(a=1\) and \(b=0.5\) in the equation. The numerical solution is computed till time $t=0.5$. The results by the explicit scheme and the semi-implicit scheme are shown in Table 1 and 2, respectively, from which we can observe the first order accuracy in time and almost second order accuracy in space (note that due to the treatment of flux shift, the spatial discretization is not exactly second order).



\begin{table}[htp]
\begin{minipage}[!t]{\columnwidth}
  \renewcommand{\arraystretch}{1.3}
  \centering
  \setlength{\tabcolsep}{0.6mm}{
  \begin{tabular}{|c|c|c|c|c|}
\hline
&\(||\omega_h-\omega_{\frac{h}{2}}||_{1}\)&
\(O_{h,L^1}\)&
\(||\omega_h-\omega_{\frac{h}{2}}||_{\infty}\)&
\(O_{h,L^\infty}\)
\\ \hline
\(h=\frac{6}{24}\)&
7.3998e-04&
1.725&
3.0168e-03&
1.633\\
\hline
\(h=\frac{6}{48}\)&
2.2377e-04&	
1.830&	
9.7252e-03&	
1.790\\
\hline
\(h=\frac{6}{96}\)&
6.2954e-05&
1.911&
2.8132e-04&
1.885\\
\hline
\(h=\frac{6}{192}\)&
1.6736e-05&
unstable&	
7.6157e-05&	
unstable\\
\hline
\(h=\frac{6}{384}\)&
unstable&
unstable&	
unstable&	
unstable\\
\hline
\(h=\frac{6}{762}\)&
unstable&
---&	
unstable&	
---\\
\hline
\end{tabular}}
  \end{minipage}
\\[12pt]
\begin{minipage}[!t]{\columnwidth}
  \renewcommand{\arraystretch}{1.3}
  \centering
  \setlength{\tabcolsep}{0.8mm}{
  \begin{tabular}{|c|c|c|c|c|}
\hline
&\(||\omega_h-\omega_{\frac{h}{2}}||_{1}\)&
\(O_{h,L^1}\)&
\(||\omega_h-\omega_{\frac{h}{2}}||_{\infty}\)&
\(O_{h,L^\infty}\)
\\
\hline
\(\tau=\frac{0.5}{1000}\)&
unstable&
unstable&
unstable&
unstable\\
\hline
\(\tau=\frac{0.5}{2000}\)&
unstable&
unstable&
unstable&
unstable\\
\hline
\(\tau=\frac{0.5}{4000}\)&
unstable&
unstable&
unstable&
unstable\\
\hline
\(\tau=\frac{0.5}{8000}\)&
1.3705e-06&
0.999&	
4.5973e-06&	
1.000\\
\hline
\(\tau=\frac{0.5}{16000}\)&
6.8523e-07&
1.000&	
2.2986e-06&	
1.000\\
\hline
\(\tau=\frac{0.5}{32000}\)&
3.4321e-07&
---&	
1.1493e-06&	
---\\
\hline
\end{tabular}}
\caption{Upper table: error with different spatial size, time step is fixed as \(\tau=\frac{0.5}{10000}\); Lower table: error with different temporal size, spatial size is fixed as \(h=\frac{6}{384}\).}
  \end{minipage}
  \vspace{-0.1cm}
\end{table}

\begin{table}[htp]
\begin{minipage}[!t]{\columnwidth}
  \renewcommand{\arraystretch}{1.3}
  \centering
  \setlength{\tabcolsep}{0.6mm}{
  \begin{tabular}{|c|c|c|c|c|}
\hline
&\(||\omega_h-\omega_{\frac{h}{2}}||_{1}\)&
\(O_{h,L^1}\)&
\(||\omega_h-\omega_{\frac{h}{2}}||_{\infty}\)&
\(O_{h,L^\infty}\)
\\ \hline
\(h=\frac{6}{24}\)&
7.3985e-04&
1.726&
3.0161e-03&
1.633\\
\hline
\(h=\frac{6}{48}\)&
2.2369e-04&	
1.830&	
9.7220e-03&	
1.790\\
\hline
\(h=\frac{6}{96}\)&
6.2910e-05&
1.912&
2.8117e-04&
1.886\\
\hline
\(h=\frac{6}{192}\)&
1.6713e-05&
1.970&	
7.6083e-05&	
1.941\\
\hline
\(h=\frac{6}{384}\)&
4.2646e-06&
2.020&	
1.9815e-05&	
1.972\\
\hline
\(h=\frac{6}{762}\)&
1.0517e-06&
---&	
5.0521e-06&	
---\\
\hline
\end{tabular}}
  \end{minipage}
\\[12pt]
\begin{minipage}[!t]{\columnwidth}
  \renewcommand{\arraystretch}{1.3}
  \centering
  \setlength{\tabcolsep}{0.8mm}{
  \begin{tabular}{|c|c|c|c|c|}
\hline
&\(||\omega_h-\omega_{\frac{h}{2}}||_{1}\)&
\(O_{h,L^1}\)&
\(||\omega_h-\omega_{\frac{h}{2}}||_{\infty}\)&
\(O_{h,L^\infty}\)
\\ \hline
\(\tau=\frac{0.5}{1000}\)&
1.0884e-05&
0.999&
3.6582e-05&
0.999\\
\hline
\(\tau=\frac{0.5}{2000}\)&
5.4426e-06&
0.999&
1.8291e-05&
0.999\\
\hline
\(\tau=\frac{0.5}{4000}\)&
2.7215e-06&
1.000&
9.1457e-06&
1.000\\
\hline
\(\tau=\frac{0.5}{8000}\)&
1.3608e-06&
1.000&	
4.5729e-06&	
1.000\\
\hline
\(\tau=\frac{0.5}{16000}\)&
6.8041e-07&
1.000&	
2.2865e-06&	
1.000\\
\hline
\(\tau=\frac{0.5}{32000}\)&
3.4021e-07&
---&	
1.1432e-06&	
---\\
\hline
\end{tabular}}
\caption{Upper table: error with different spatial size, time step is fixed as \(\tau=\frac{0.5}{10000}\); Lower table: error with different temporal size, spatial size is fixed as \(h=\frac{6}{384}\).
}
  \end{minipage}
  \vspace{-0.1cm}
\end{table}

Though the difference between numerical results of the semi-implicit scheme and the explicit scheme is small when both methods are stable, it is obvious that the semi-implicit scheme does not suffer from the parabolic type stability constraint. So in the rest of this section, we only apply the semi-implicit scheme.

\subsection{Global solution and blow-up}
\label{longtime}

In this subsection, we focus on global solutions and blow-up phenomena. We verify the relative entropy property proved in section 3. All the tests are done through the semi-implicit scheme with spatial size \(h=0.02\) and temporal size \(\tau=10^{-3}\).

1) Blow-up

As shown in \cite{CCP2011}, for an excitatory system ( \(b>0\) ), the solution may blow up in finite time for certain initial conditions. In Theorem 2.2 of \cite{CCP2011}, it is found  that the density function blows up due to large connectivity and initial condition that concentrates at \(V_F\). Figure \ref{blowup} exhibits this phenomena. Intuitively, the density function results in blow-up because it becomes much more concentrated at endpoint \(V_R\) as time involves. Also we can see as the firing rate increases, the solution at \(V_R\) becomes more and more steep.

\begin{figure}
\centering
\includegraphics[width=0.48\textwidth]{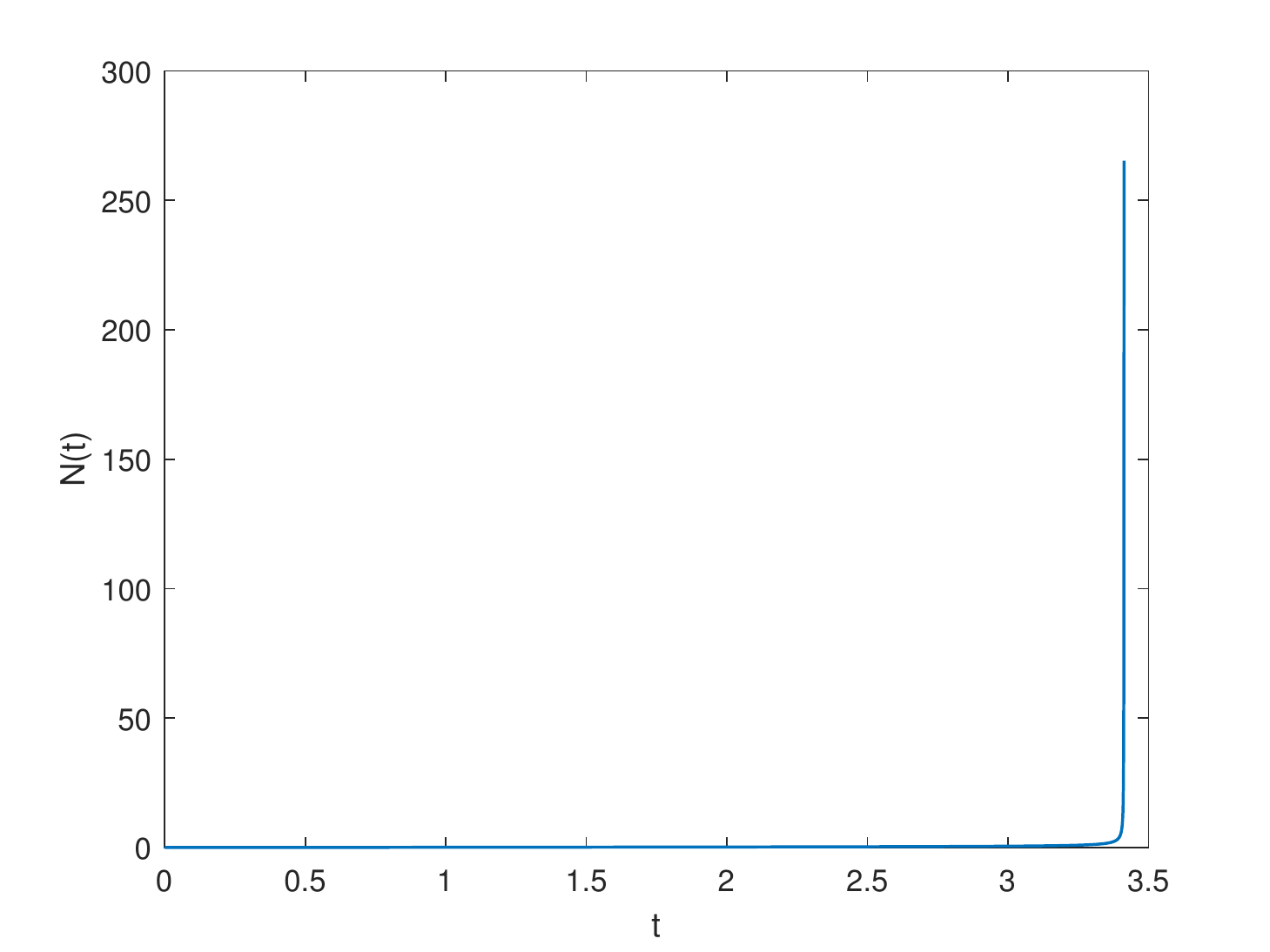}
\includegraphics[width=0.48\textwidth]{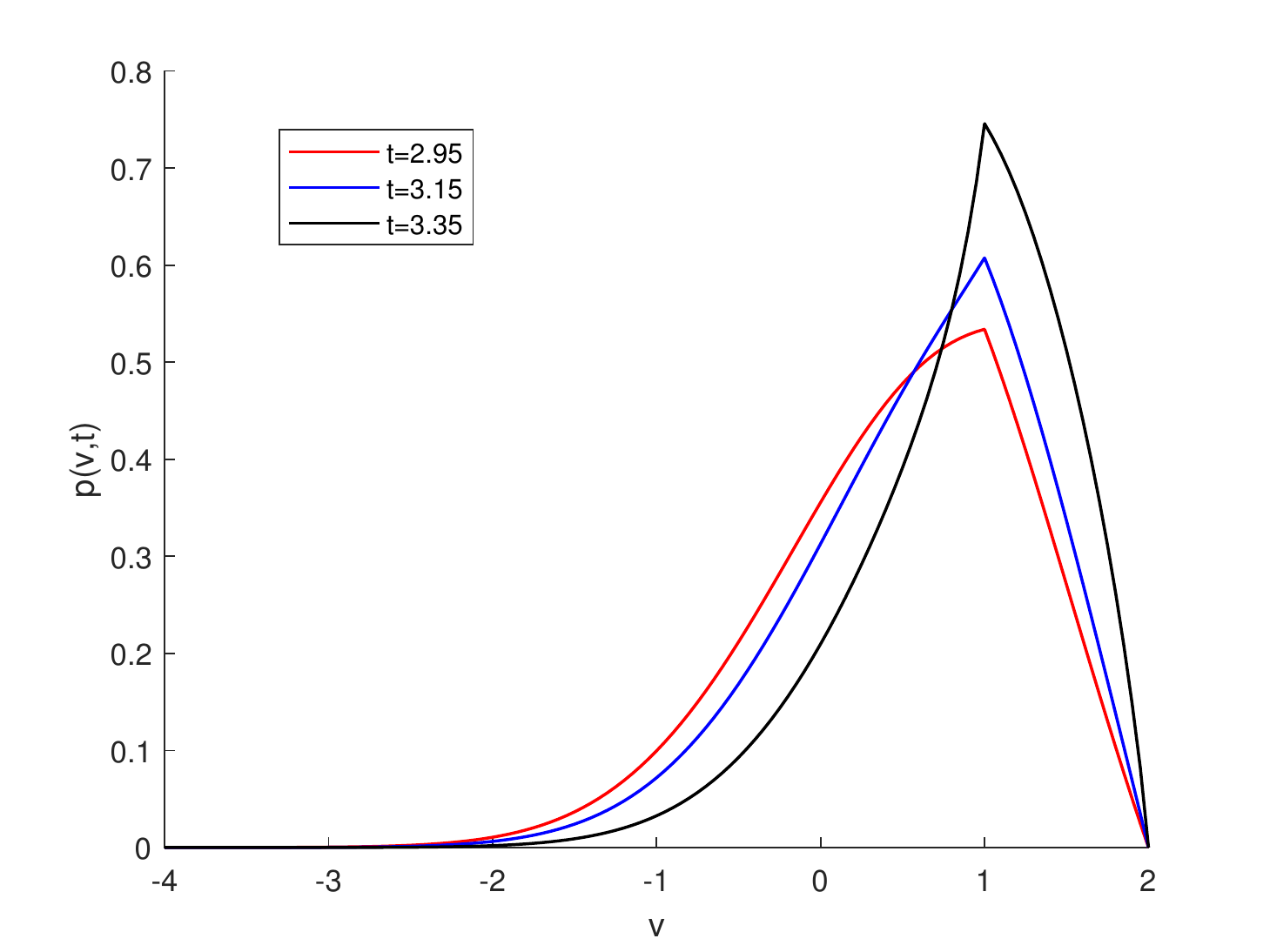}
\includegraphics[width=0.48\textwidth]{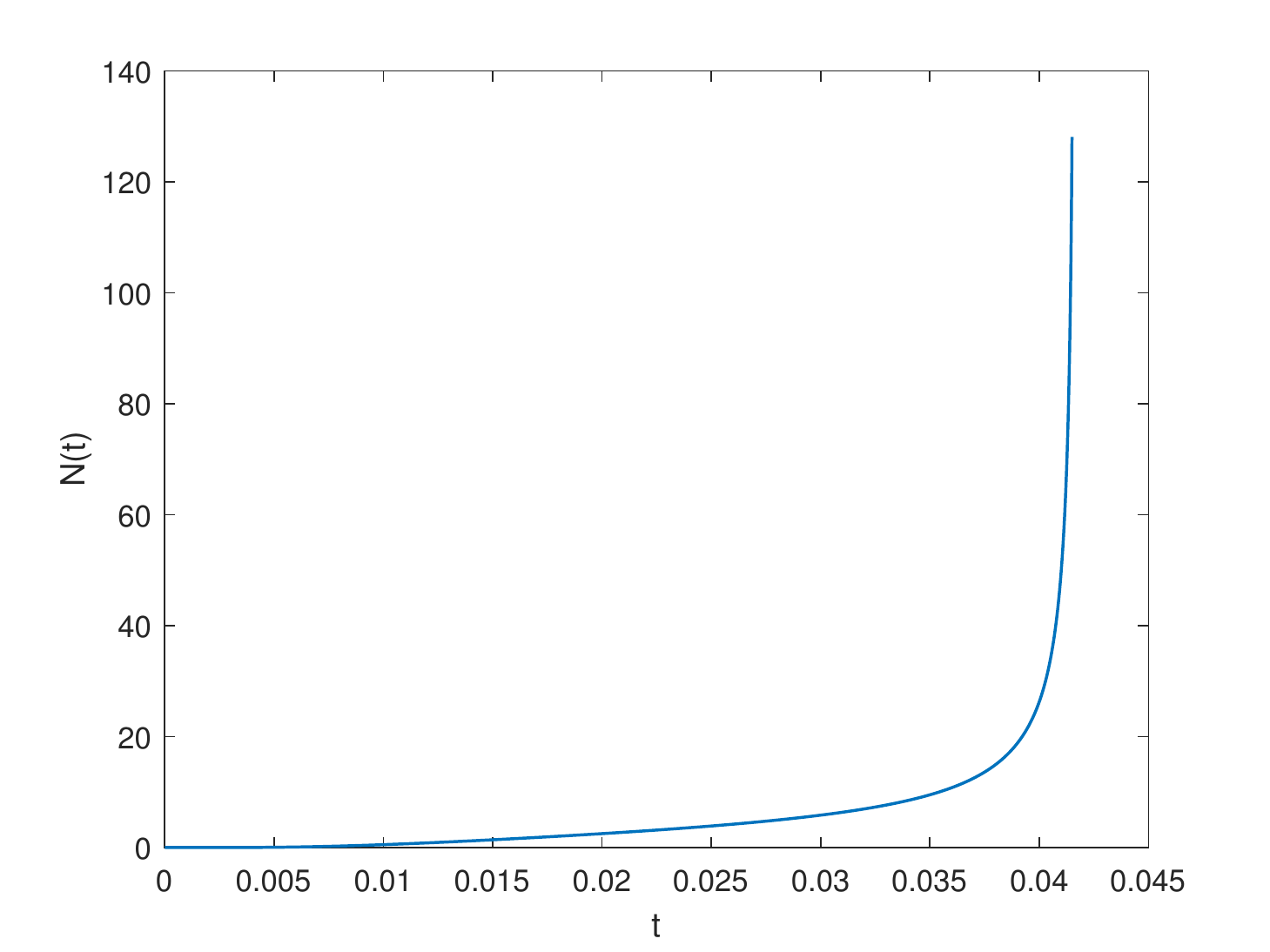}
\includegraphics[width=0.48\textwidth]{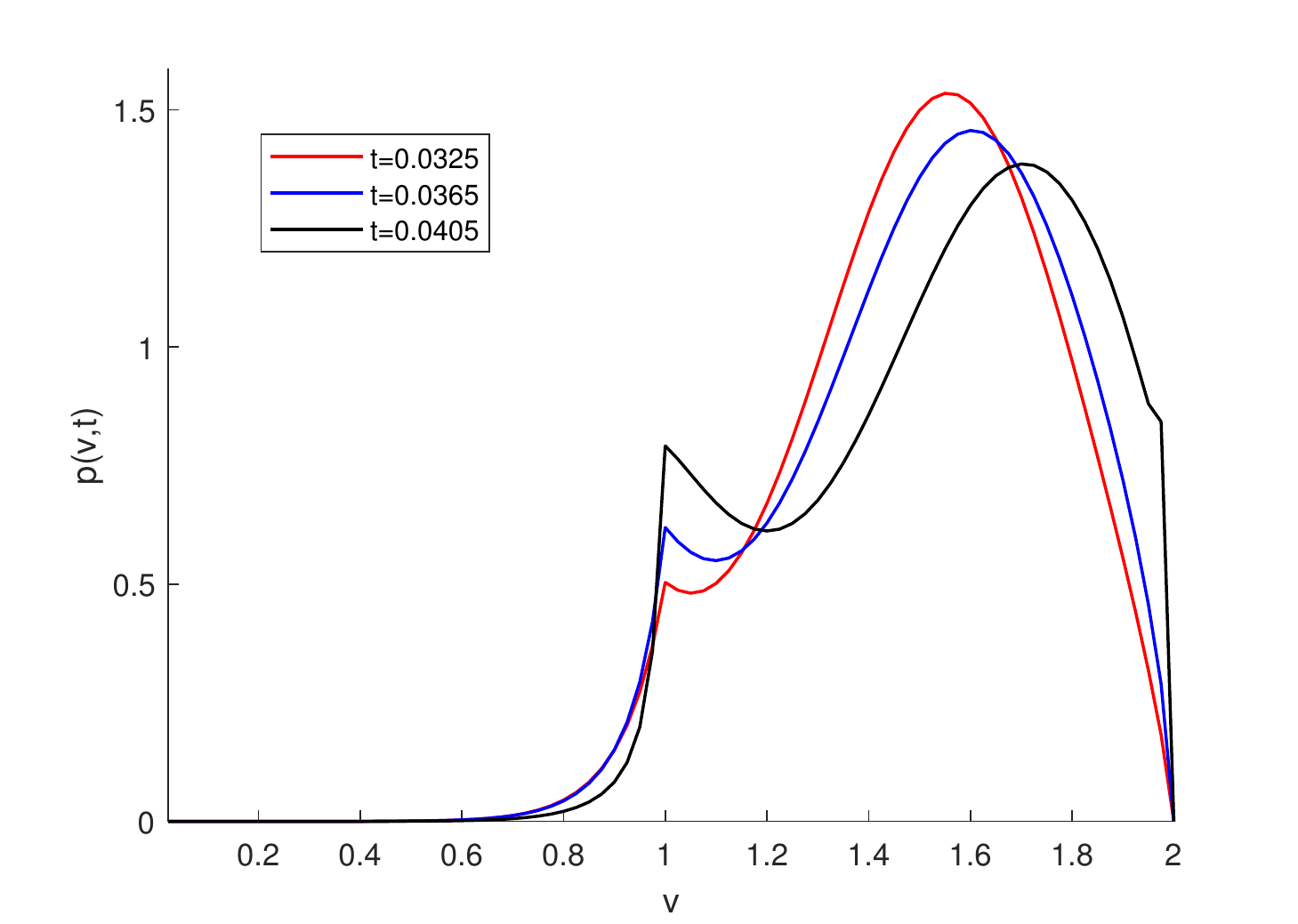}
\caption{(Excitatory blow-up) Top: Equation parameters \(a(N(t))\equiv1\), \(b=3\) with Gaussian initial condition \(v_0=-1\), \(\sigma_0^2=0.5\). Top left: evolution of firing rate \(N(t)\). Top right: density at \(t=2.95,3.15,3.35\). Bottom: Equation parameters \(a(N(t))\equiv1\), \(b=1.5\) with Gaussian initial condition \(v_0=1.5\), \(\sigma_0^2=0.005\). Bottom left: evolution of firing rate \(N(t)\). Bottom right: density at \(t=0.0325,0.0365,0.0405\). }
\label{blowup}
\end{figure}

2) Stationary solutions

As shown in \cite{CCP2011}, there may be zero, one or two stationary solutions for the system. In our tests, we focus on several stationary solutions scenarios and test the stability of the stationary solutions.

For example, when \(a(N(t))\equiv1\) and \(b=1.5\), we can find two different steady states satisfying Equation \eqref{eq:ss} whose firing rates are \(N^{\infty}=2.319\) and \(N^{\infty}=0.1924\). Let the two steady states be the initial condition (stationary initial condition see Equation \eqref{ICstationary}), we study the evolution of density functions as time goes by in Figure \ref{stationarysolution}, from which we can see that the stationary solution with firing rate \(N^{\infty}=2.319\) is unstable while the steady state with firing rate \(N^{\infty}=0.1924\) is stable. The former doesn't change much in a short time but converts to the stable stationary state later.

\begin{figure}
\centering
\includegraphics[width=0.48\textwidth]{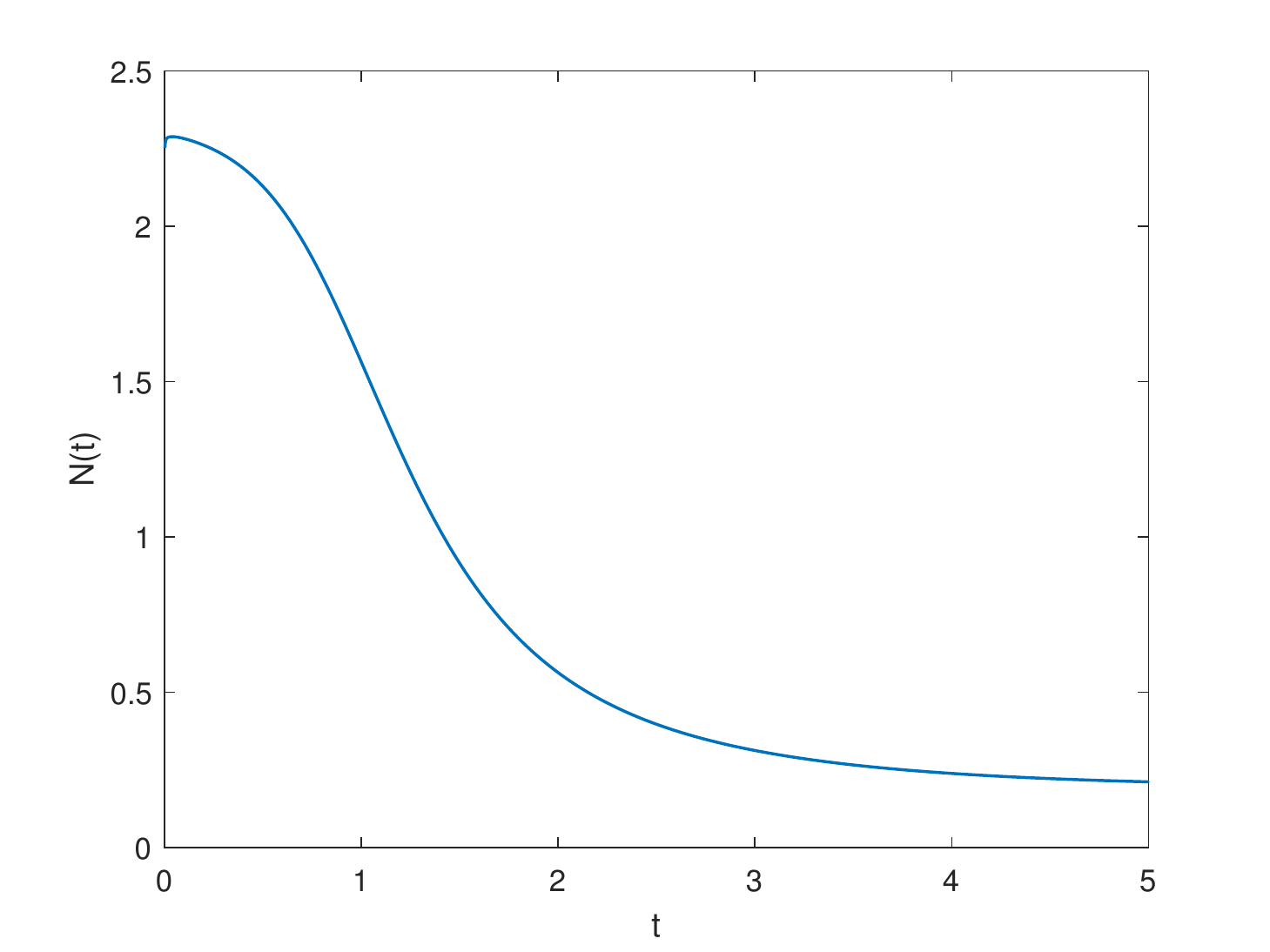}
\includegraphics[width=0.48\textwidth]{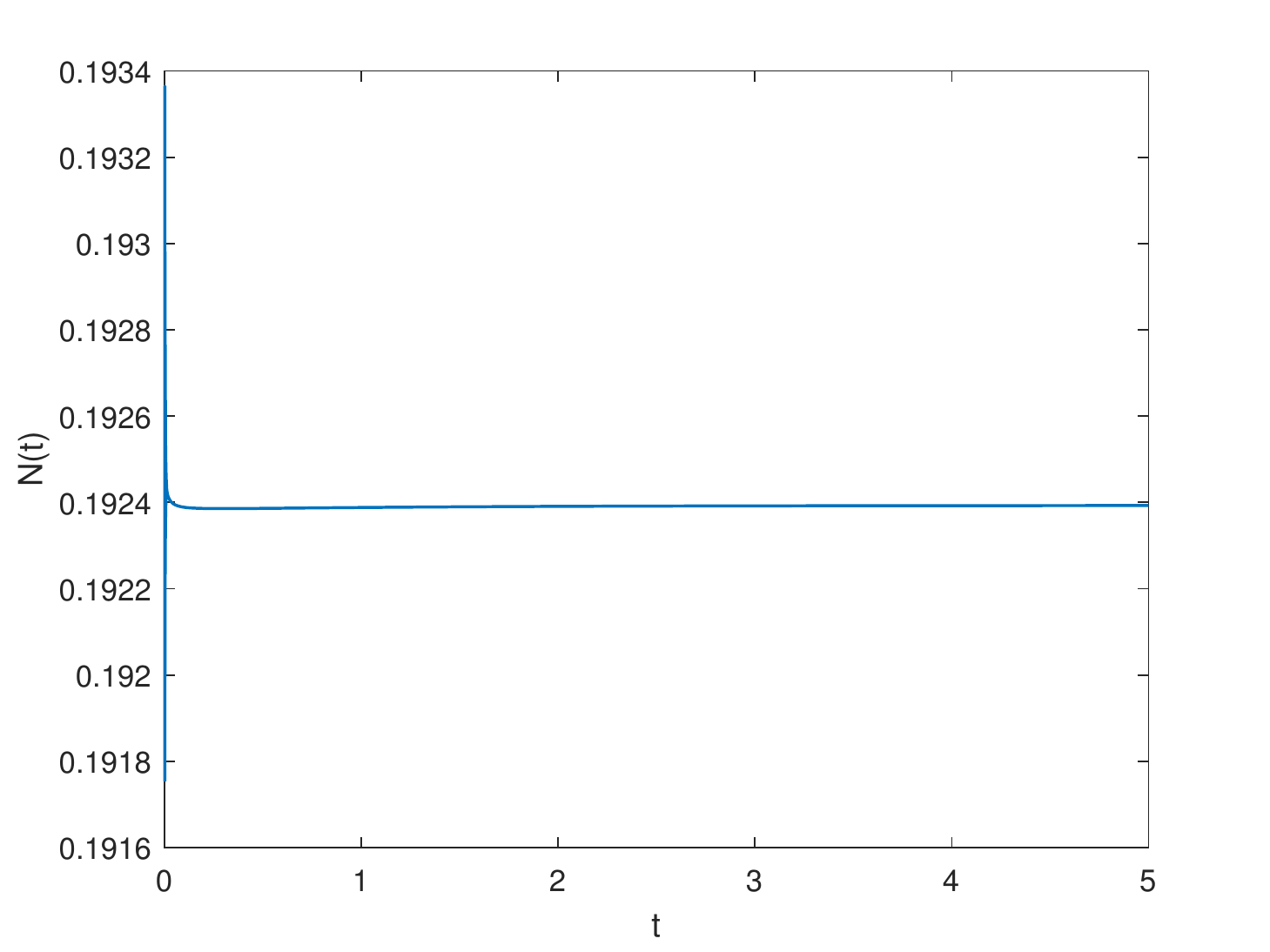}
\includegraphics[width=0.48\textwidth]{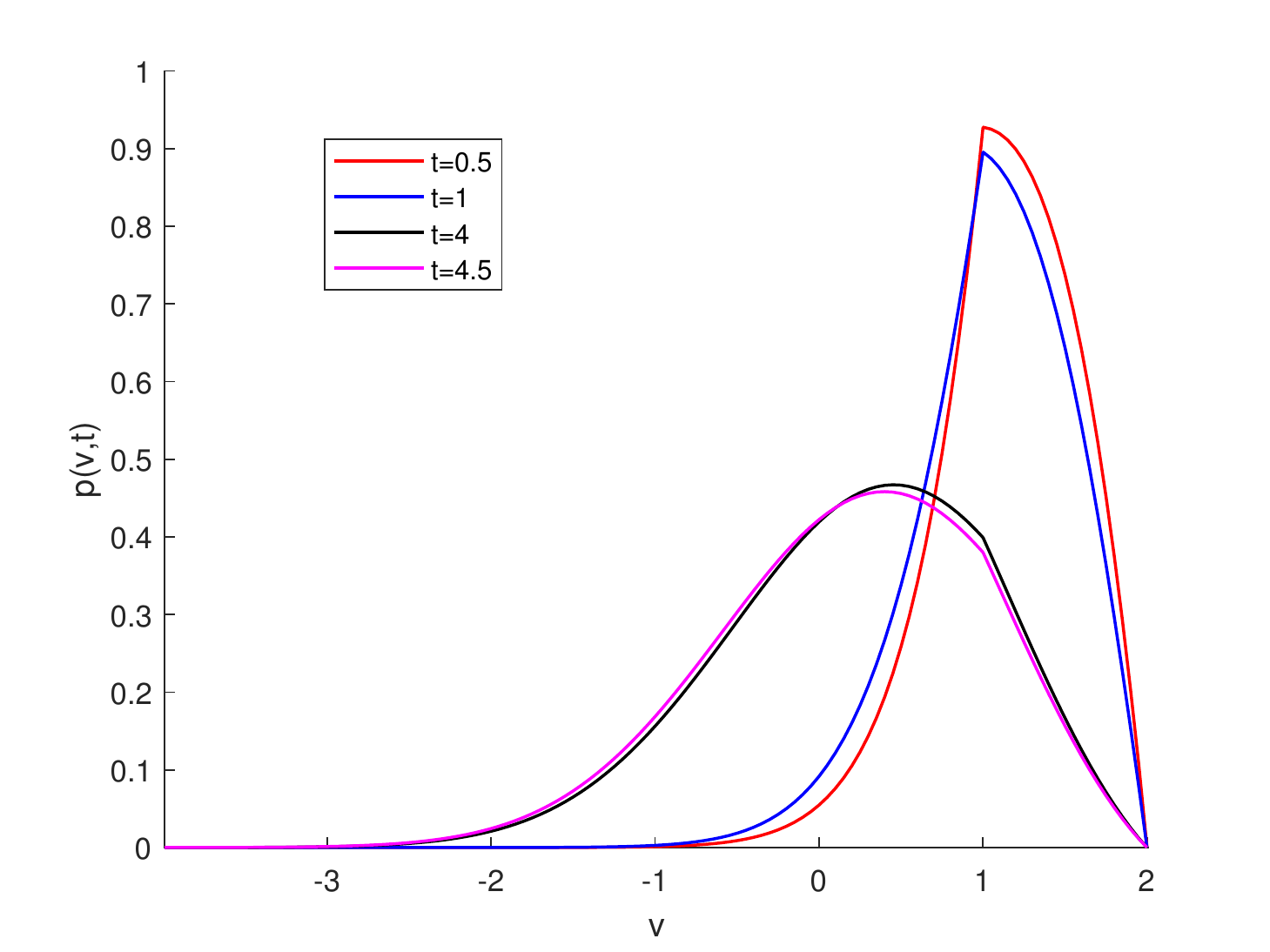}
\includegraphics[width=0.48\textwidth]{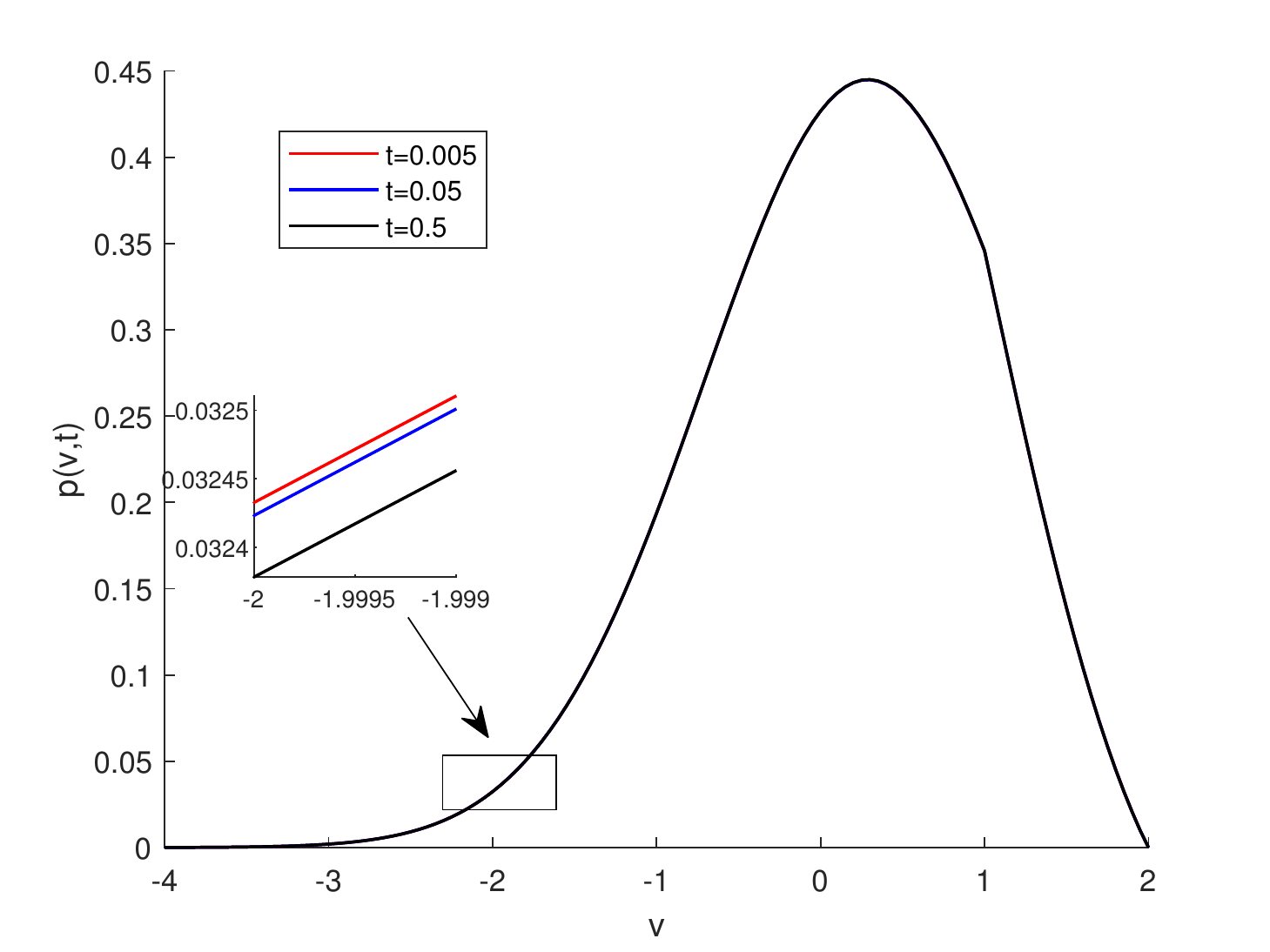}
\caption{{(Stationary solutions) Equation parameters \(a \equiv1\), \(b=1.5\). In this case we find two different stationary solutions, thus we can choose initial conditions as stationary states (see equation \eqref{ICstationary}) with different firing rates. Top left: evolution of the firing rate \(N(t)\) for stationary state initial condition with  \(N^{\infty}_1=2.319\). Top right: evolution of the firing rate \(N(t)\) for stationary state initial condition \(N^{\infty}_2=0.1924\). Bottom left: density for initial condition \(N^{\infty}_1=2.319\) at \(t=0.5,1,4.4.5\). Bottom right: density for initial condition \(N^{\infty}_2=0.1924\) at \(t=0.005,0.05,0.5\).}}
\label{stationarysolution}
\end{figure}

3) Relative entropy

Now we verify Theorem \ref{thm:nre} in Section 3.2 through numerical tests. Note that the value of the stationary solution we use to evaluate the relative entropy are from the exact expression given by Equation \eqref{stationary} rather than \(p_i^\infty\) solved by Equation \eqref{nss}.

First, we consider a linear case with \(a(N(t))=1\) and \(b=0\). It has a unique stationary solution with firing rate \(N^{\infty}=0.1377\) (for uniqueness proof see \cite{CCP2011}). Figure \ref{REforlinear} shows the time evolution of the firing rate and relative entropy for this case.

\begin{figure}
\centering
\includegraphics[width=0.48\textwidth]{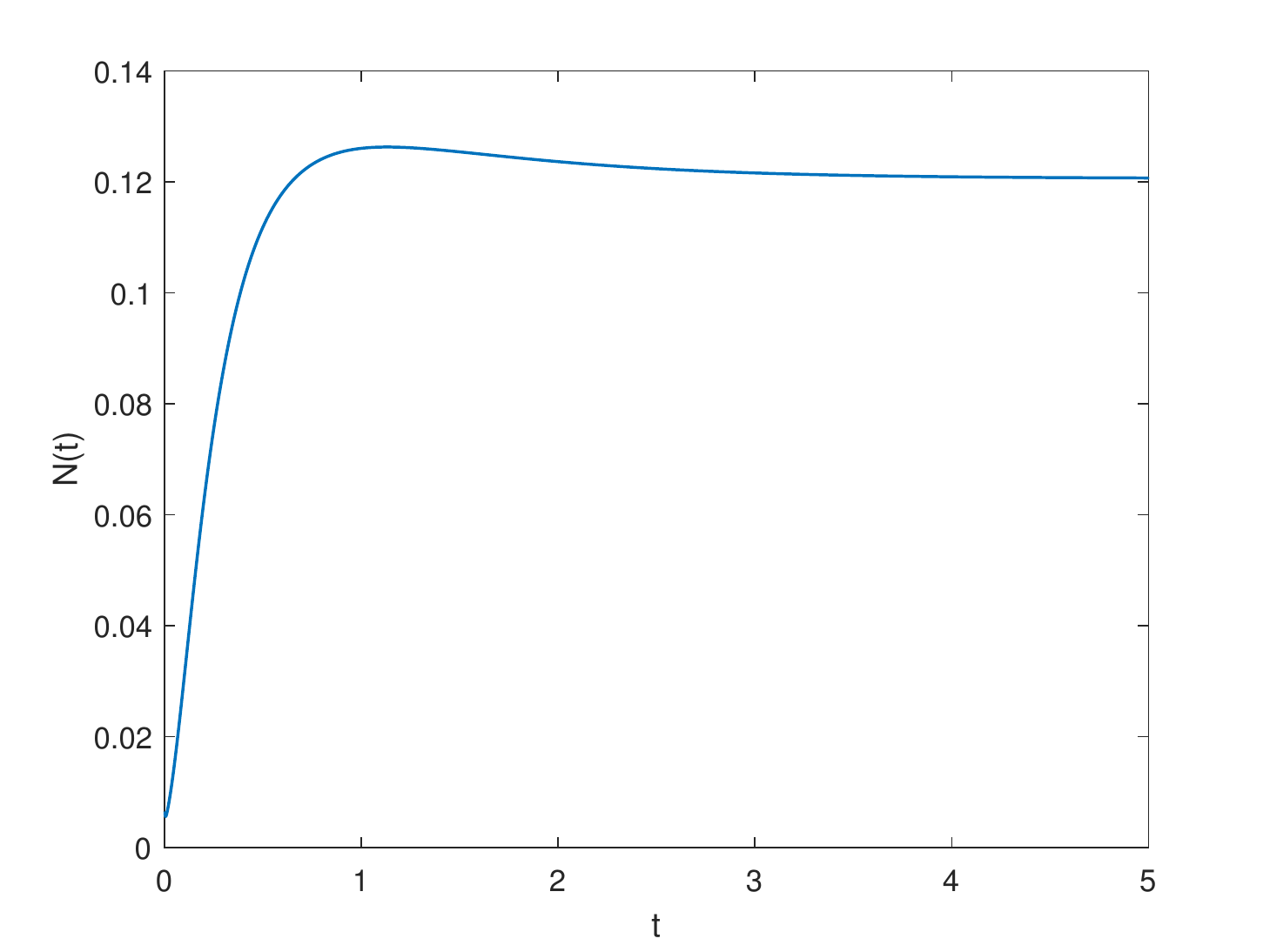}
\includegraphics[width=0.48\textwidth]{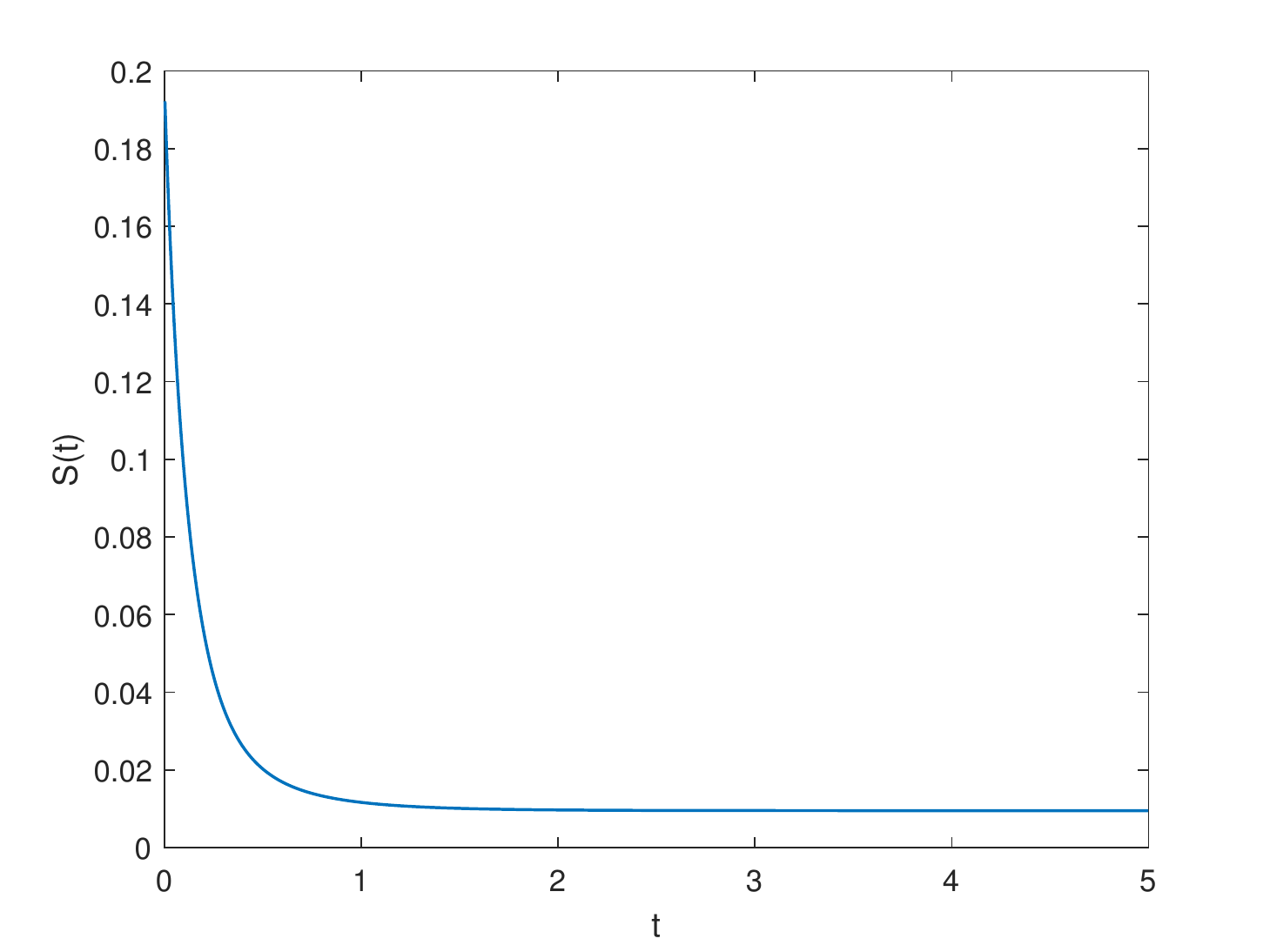}
\caption{(Decay of relative entropy for the linear case) Equation parameters  \(a\equiv1\), \(b=0\) with Gaussian initial condition \(v_0=0\), \(\sigma_0^2=0.25\). In this case we find a unique stationary solution with firing rate \(N^{\infty}=0.1377\). Left: firing rate \(N(t)\). Right: relative entropy \(S(t)\) with \(G(x)=\frac{(x-1)^2}{2}\).}
\label{REforlinear}
\end{figure}

\begin{figure}
\centering
\includegraphics[width=0.48\textwidth]{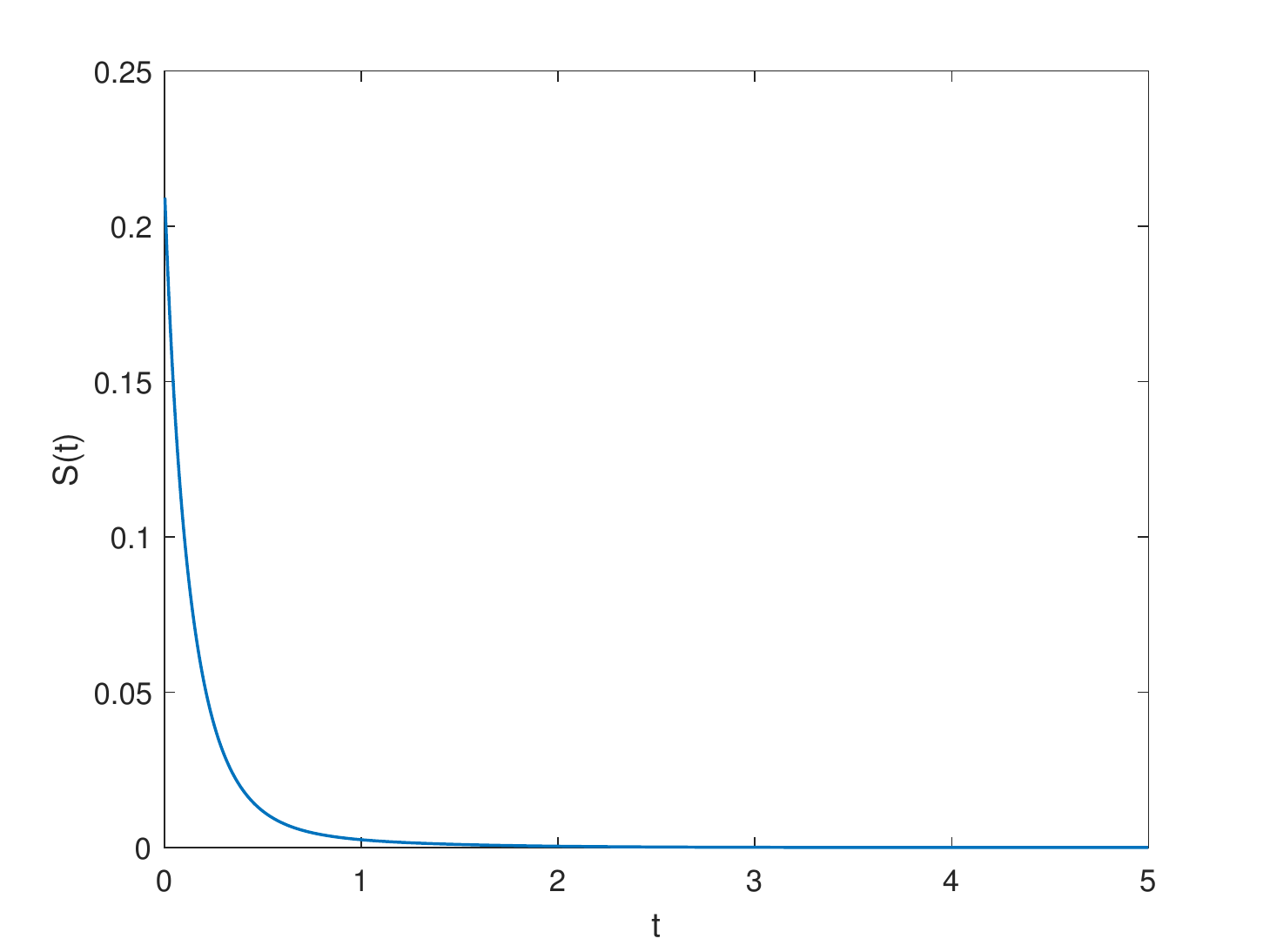}
\includegraphics[width=0.48\textwidth]{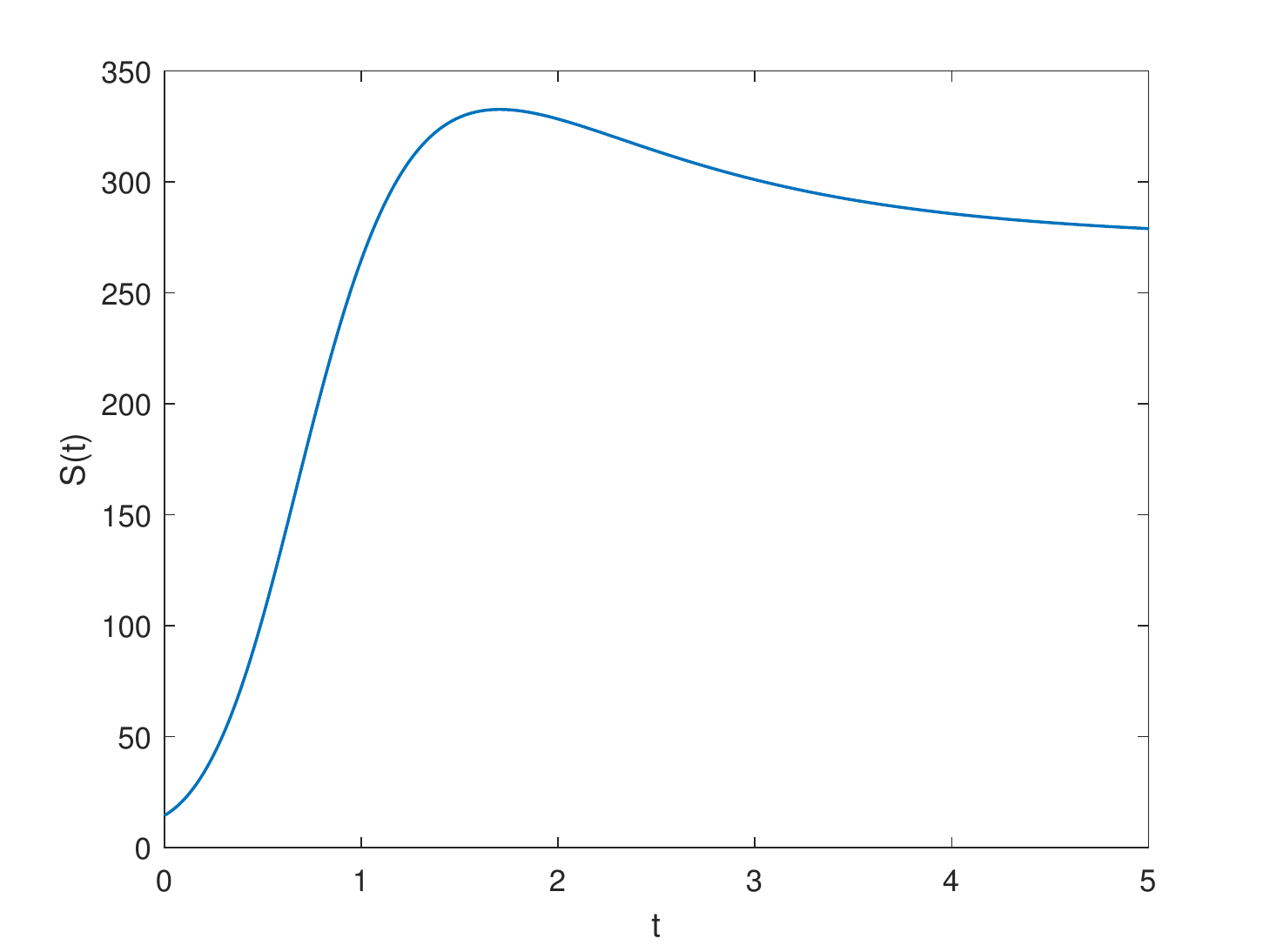}
\caption{(Evolution of relative entropy for the two stationary states case) Equation parameters  \(a(N(t))\equiv1\), \(b=1.5\) with Gaussian initial condition \(v_0=0\), \(\sigma_0^2=0.25\). In this case we find two stationary states with firing rate \( N^{\infty}=0.1924\) and \( N^{\infty}=2.319\), thus we can define two relative entropy according to the two stationary solutions (see Figure \ref{stationarysolution} for full discussion). Left: relative entropy \(S(t)\) according to stable stationary state with \( N^{\infty}=0.1924\) with \(G(x)=\frac{(x-1)^2}{2}\). Right: relative entropy \(S(t)\) according to unstable stationary state with \( N^{\infty}=2.319\) with \(G(x)=\frac{(x-1)^2}{2}\). }
\label{REfornonlinear}
\end{figure}

Then we move to excitatory system case when \(b>0\). We consider the same example as in Figure \ref{stationarysolution}, where we see two stationary solutions. We can write relative entropy according to each stationary solution. The results are shown in Figure \ref{REfornonlinear}, where the density function converges to the stable stationary state with \(N^{\infty}=0.1915\). We see the the relative entropy according to this stable state decreases, while the other one doesn't.

Finally we consider a case when \(a_1\neq0\). We choose \(a_0=1\), \(a_1=0.1\) and \(b=0\) and we also choose the Gaussian function as the initial condition. We {find} a stationary solution with firing rate \(N^{\infty}=0.1420\). The relative entropy still decreases as shown in Figure \ref{REfora1neq0}.

{\begin{remark}
We remark that the decreasing relative entropy case in Figure \ref{REfornonlinear} doesn't indicate invalidity of Theorem \ref{thm:nre}. Since we see the numerical solutions converge to the stable stationary solution with \( N^{\infty}=0.1924\) in Figure \ref{stationarysolution} rather than the unstable one, we do not expect the relative entropy with respect to the unstable state decrease. 
\end{remark}}

\begin{figure}
\centering
\includegraphics[width=0.48\textwidth]{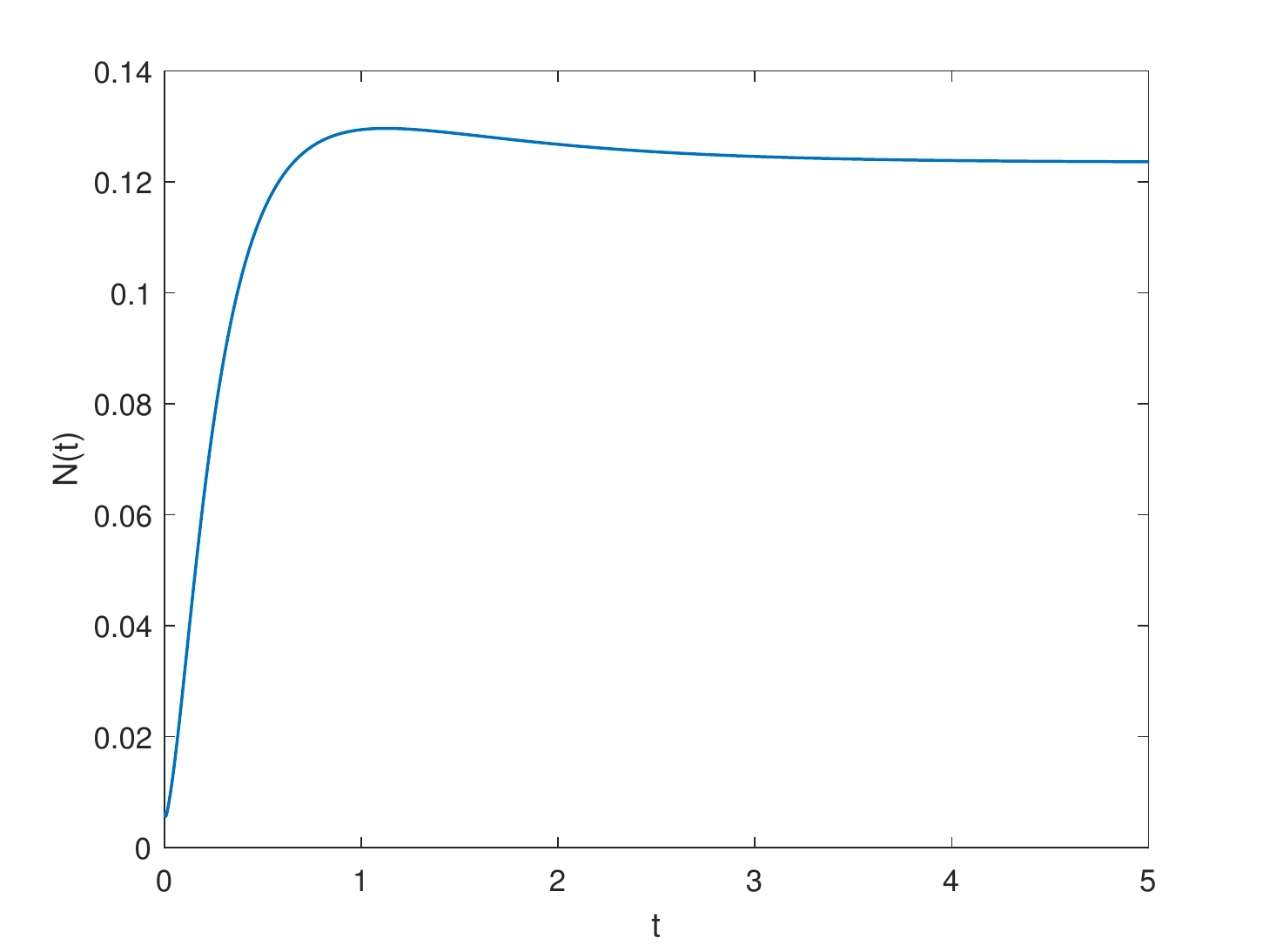}
\includegraphics[width=0.48\textwidth]{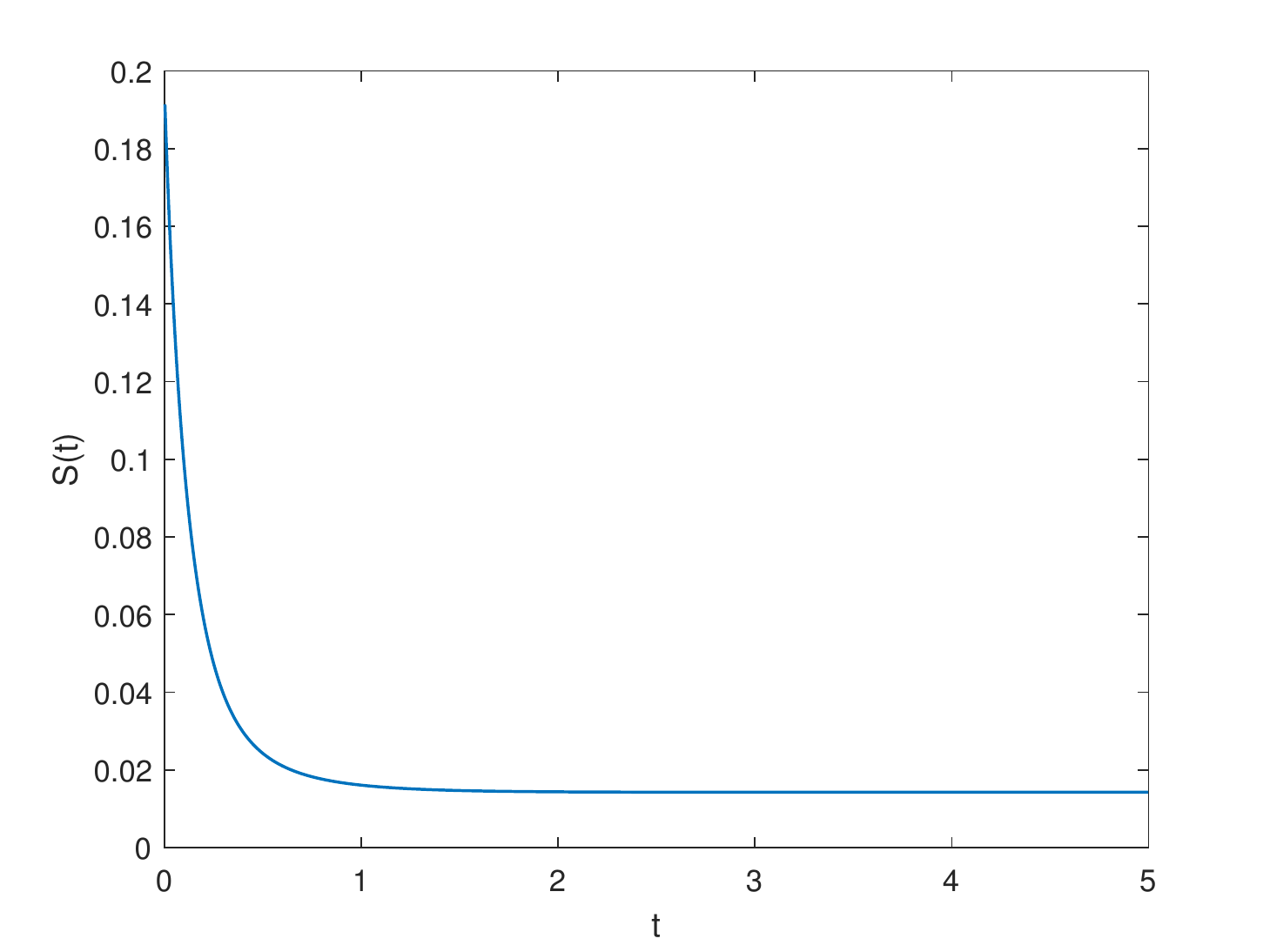}
\caption{(Decay of relative entropy for nonlinear case when \(a_1\neq0\)) Equation parameters \(a_0=1\),  \(a_1=0.1\), \(b=0\) with Gaussian initial condition \(v_0=0\), \(\sigma_0^2=0.25\). We find a stationary solution with firing rate \(N^{\infty}=0.1420\). Left: firing rate \(N(t)\). Right: relative entropy \(S(t)\)  with \(G(x)=\frac{(x-1)^2}{2}\).}
\label{REfora1neq0}
\end{figure}

\newpage

\subsection{Modified NNLIF model and numerical studies}
\label{modify}

We consider the Fokker-Planck model involving the transmission delay and the refractory state (introduced in \cite{CS2018}):
\begin{equation} \label{meq:fk}
\begin{cases}
\partial_t p+\partial_v ( h(v,N(t-D)) p) - a(N(t-D)) \partial_{vv}p(v,t)
=0, \quad v\in(-\infty,V_F)/\{V_R\},  \\
N(t)=-a(N(t-D))\partial_v p(V_F,t),\\
\frac{d}{dt}R(t)=N(t)-\frac{R(t)}{\gamma},\\
p(v,0)=p_0(v),\qquad p(-\infty,t)=p(V_F,t)=0, \\
p(V_R^-,t)=p(V_R^+,t),\qquad a(N(t-D))\frac{\partial}{\partial v}p(V_R^-,t)=a(N(t-D))\frac{\partial}{\partial v} p(V_R^-,t)+\frac{R(t)}{\gamma},
\end{cases}
\end{equation}
where \(D\) indicates the time of transmission delay for the firing rate, \(R(t)\) indicates the proportion of neuron in the refractory state at time \(t\) with refractory period \(\gamma\) and \(N(t)\) denotes the  mean firing rate. Compared to Equation \eqref{eq:fk}, the system \eqref{meq:fk} is a PDE of the density function \(p(v,t)\) with an ODE involved with refractory state \(R(t)\). Here we consider
\begin{equation} \label{meq:parameter}
h(v,N(t))=-v+bN(t)+v_{\textrm{ext}}, \quad a(N(t))=a_0+a_1N(t),
\end{equation}
where \(v_{\textrm{ext}}\) describes the external synapses. In Equation \eqref{meq:fk}, \(v_{\textrm{ext}}\) carries out a drift in the flux.

Note that the initial condition of Equation \eqref{meq:fk} is chosen as
\[
\int_{-\infty}^{V_F}p(v,0)dv+R(0)=1.
\]
And it is easy to check
\[
\int_{-\infty}^{V_F}p(v,t)dv+R(t)=1,
\]
for any \(t>0\).

We introduce a fully-discrete finite difference scheme for  system \eqref{meq:fk}, which is natural extension of the semi-implicit scheme  proposed in section 2. The way we define the numerical flux for the PDE  in \eqref{meq:fk} follows the same semi-implicit scheme, which is straightforward. To discretize the ODE of $R(t)$, we choose the forward Euler scheme:
\begin{equation}\label{ODEsolver}
\frac{R_h^{m+1}-R_h^m}{\tau}=N_h^m-\frac{R_h^m}{\gamma}.
\end{equation}

The reason of using  the explicit scheme for the ODE is because it naturally preserves the total mass when we choose \(F_{n-\frac{1}{2}}^m=N^m\). In other words,
\[
\left(h\sum_{i=1}^np_i^{m+1}+R_h^{m+1}\right)-\left(h\sum_{i=1}^np_i^m+R_h^m\right)=0.
\]
If we consider the backward Euler scheme for $R(t)$ as
\[
\frac{R_h^{m+1}-R_h^m}{\tau}=N_h^{m+1}-\frac{R_h^{m+1}}{\gamma},
\]
it does not preserve the discrete mass because
\[
\left(h\sum_{i=1}^np_i^{m+1}+R_h^{m+1}\right)-\left(h\sum_{i=1}^np_i^m+R_h^m\right)=\frac{\tau}{\gamma}(R_h^m-R_h^{m+1})+\tau(N_h^{m+1}-N_h^m).
\]
More sophisticated mass-preserving time discretization of the refractory state  will be considered in future work.


As shown in \cite{CS2018}, oscillatory solutions appears to exist when initial conditions concentrate around \(V_F\) or \(v_{\textrm{ext}}\) is large. So we change our computational domain from \([-4,2]\) to \([0,2]\) in this subsection to show phenomenon near \(V_F\) better.

First we consider inhibitory system  (\(b<0\)). The figures in the left column of Figure \ref{fig:inhibitosc} shows the periodic solution when \(v_{\textrm{ext}}\) is large. In fact, large external synapses \(v_{\textrm{ext}}\) also results in density function concentrating on \(V_F\). We also find that when \(v_{\textrm{ext}}\) is not large enough, the oscillation solution may decay to stationary solution. 

The figures in the right column of in Figure \ref{fig:inhibitosc} seem to suggest that the transmission delay \(D\) also plays an important role for oscillatory solution. When \(D\) is too small, periodic solution may not exist. {When \(D\) is large, the frequency of periodic solution decreases.}

\begin{figure}
\centering
\includegraphics[width=0.48\textwidth]{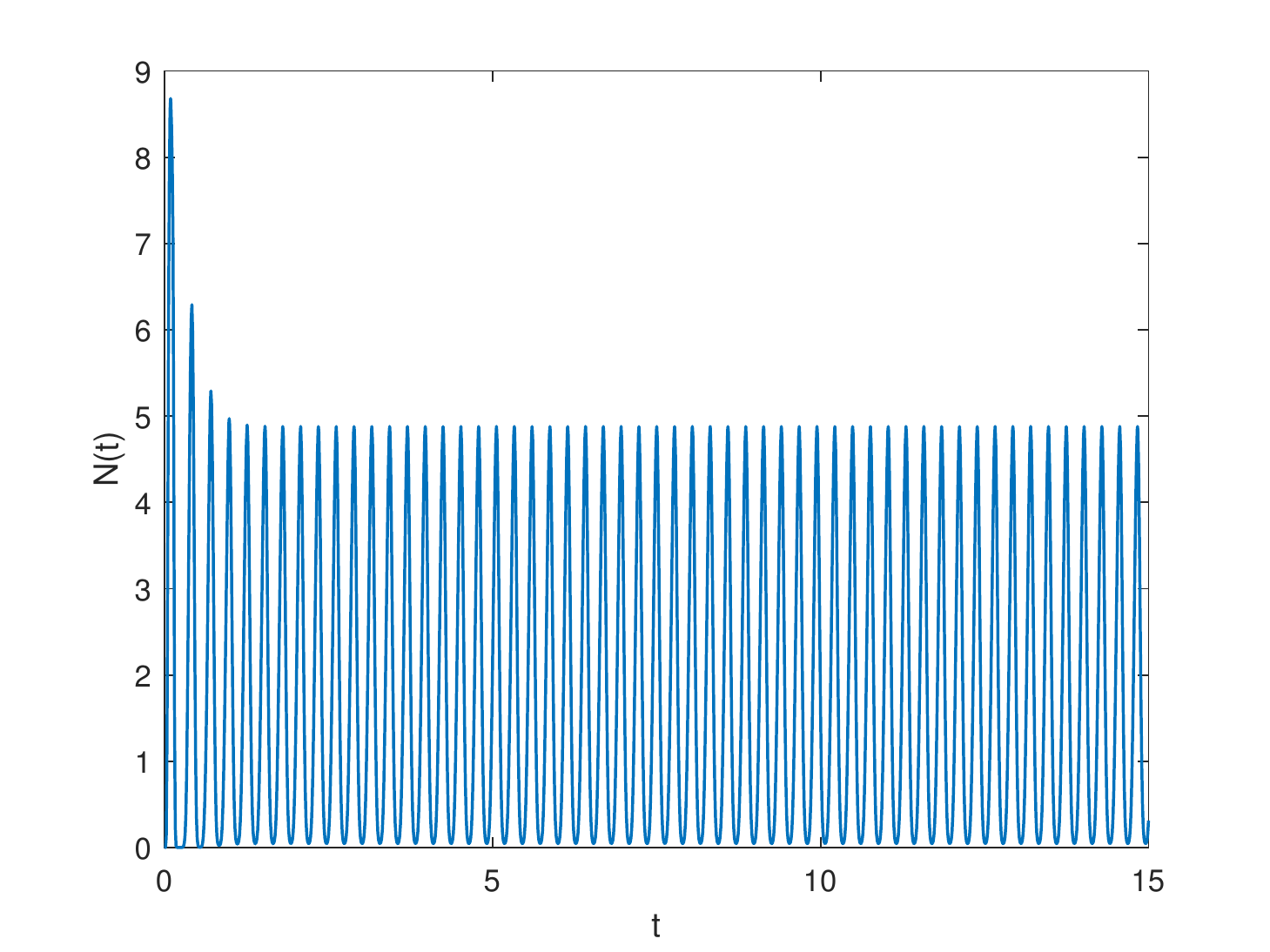}
\includegraphics[width=0.48\textwidth]{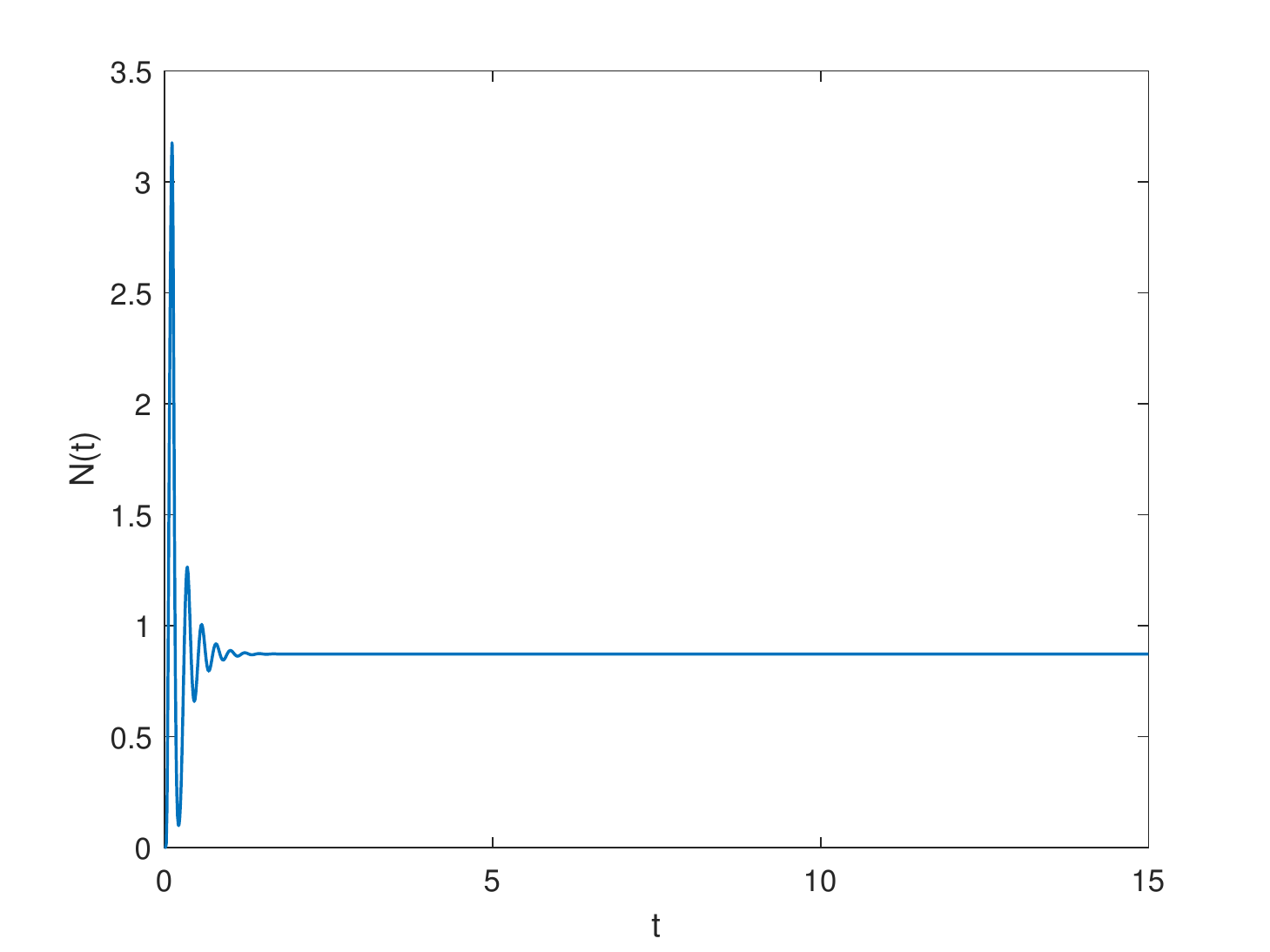}
\includegraphics[width=0.48\textwidth]{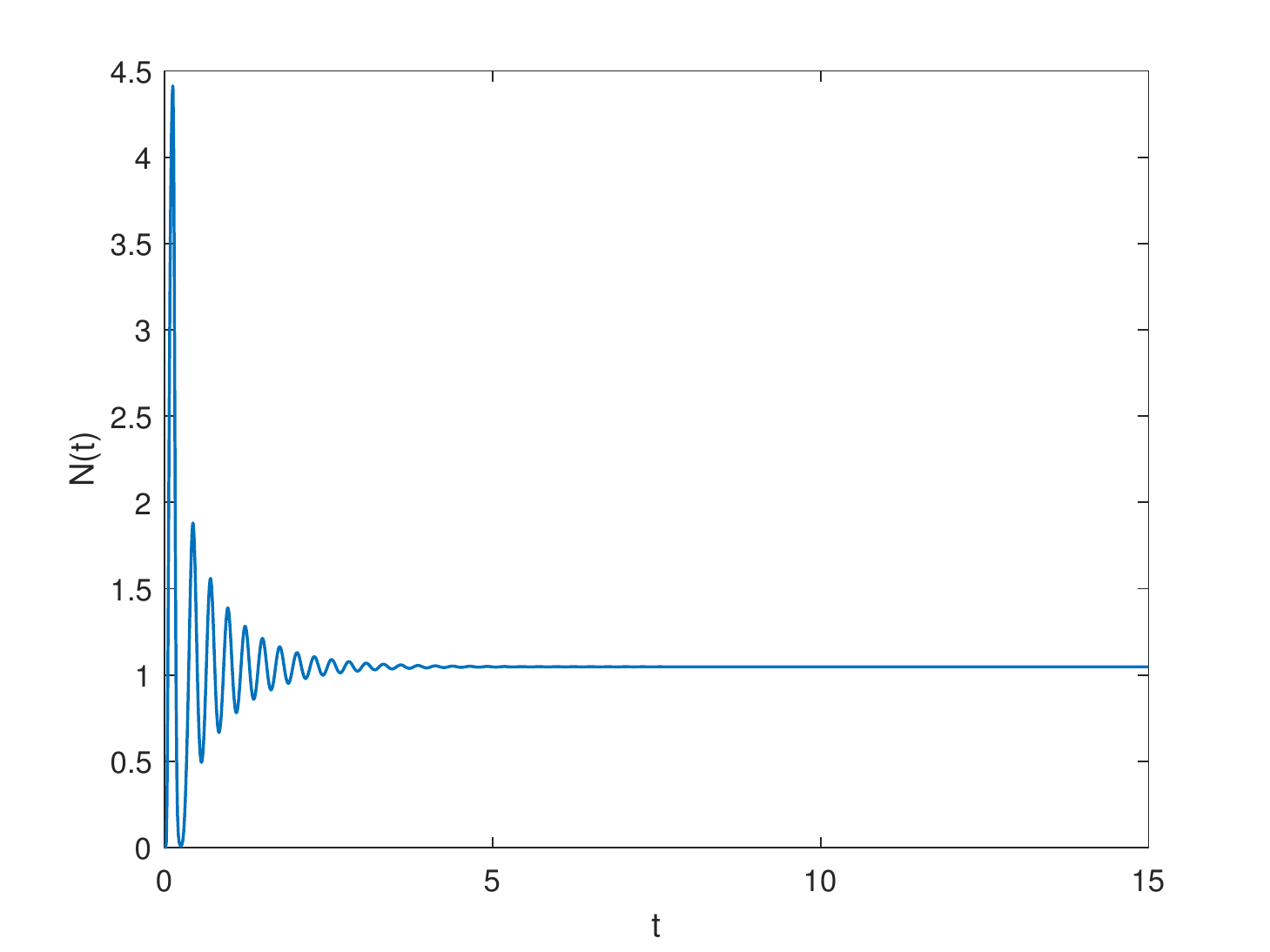}
\includegraphics[width=0.48\textwidth]{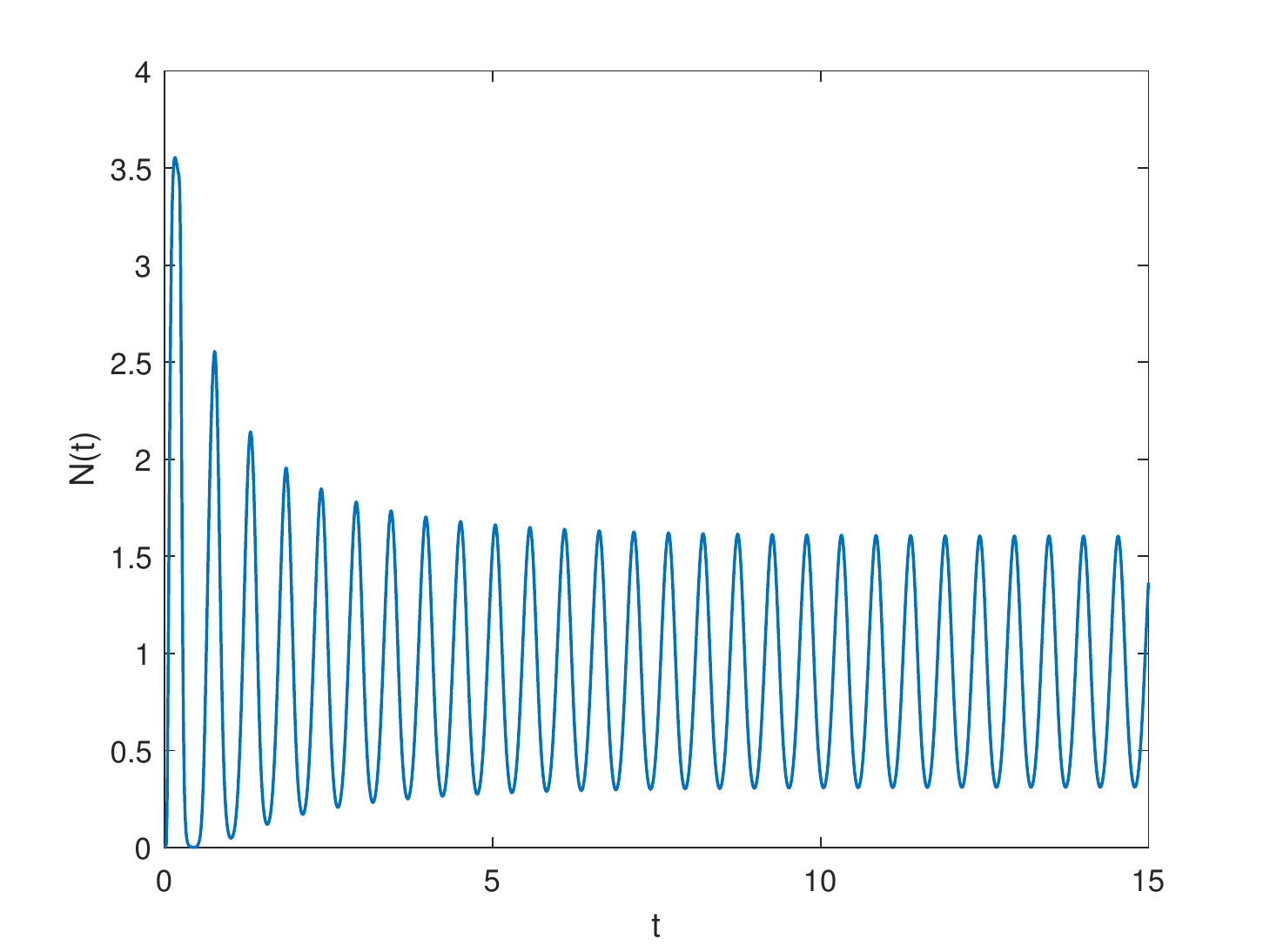}
\includegraphics[width=0.48\textwidth]{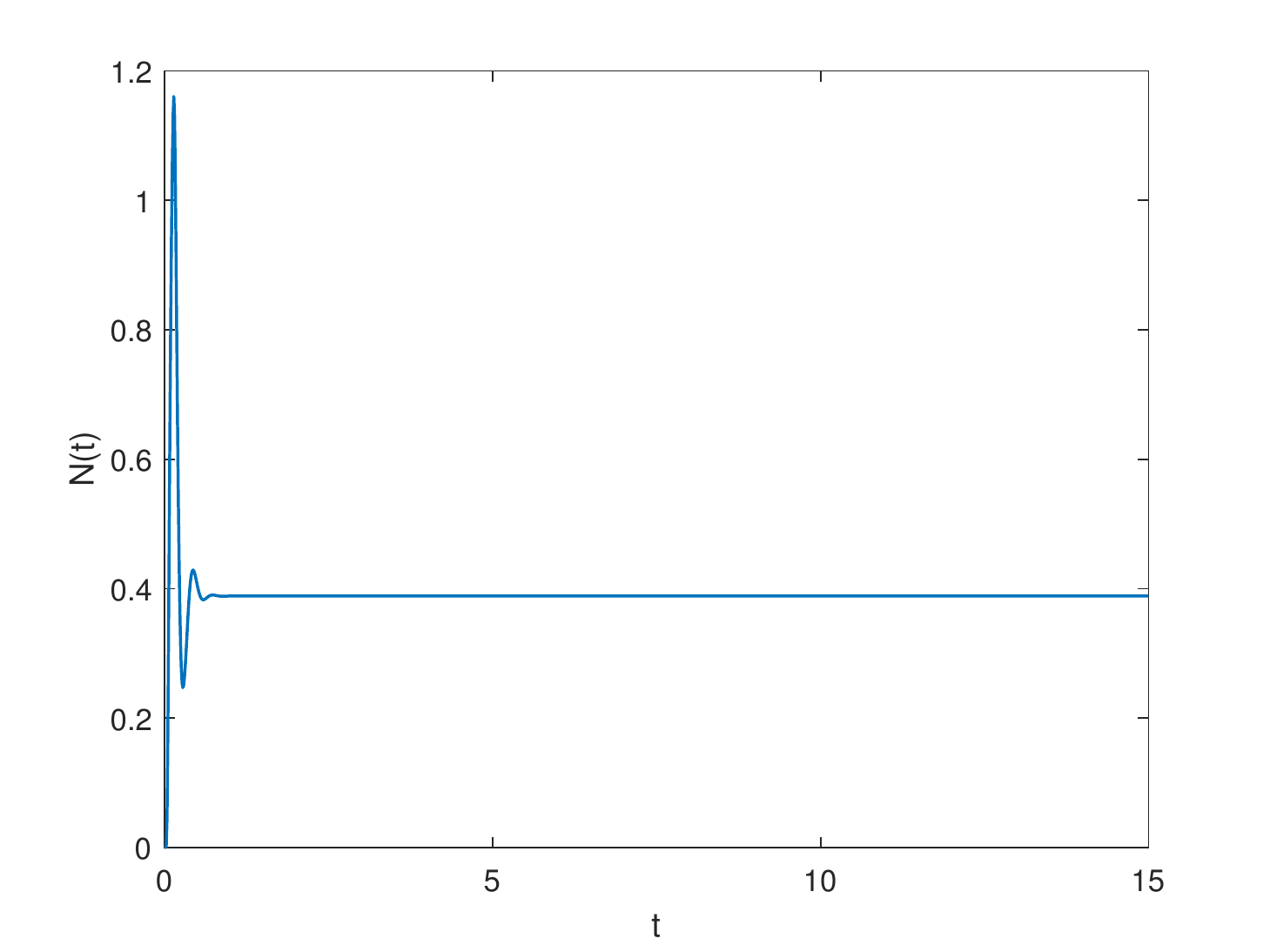}
\includegraphics[width=0.48\textwidth]{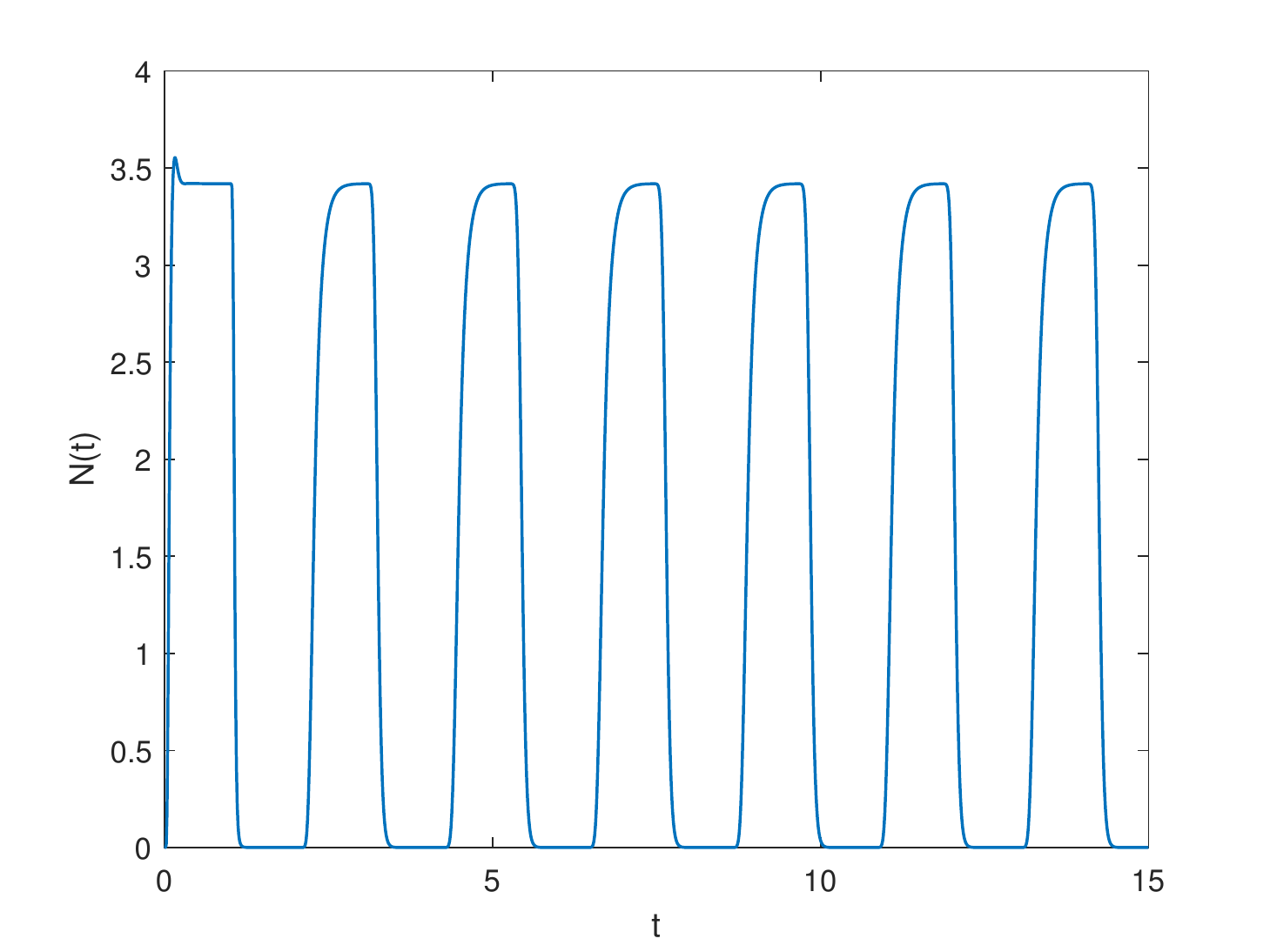}
\caption{{(Inhibitory periodic solution of model with transmission delay and refractory state) Equation parameters: \(a(N(t))\equiv 1\) and \(b = -4\). Gaussian initial data: \(v_0=1\) and \(\sigma_0\)= 0.0003. Refractory states: \(\gamma = 0.025\) and \(R(0) = 0.2\). We choose the spatial size \(h=\frac{2}{60}\) and the temporal size \(\tau=2\times 10^{-3}\) for the tests. Top left:  Transmission delay \(D = 0.1\), external synapses \(v_{\textrm{ext}}=10\); Middle left:  Transmission delay \(D = 0.1\), external synapses \(v_{\textrm{ext}}=6\); Bottom left:  Transmission delay \(D = 0.1\), external synapses \(v_{\textrm{ext}}=2\); Top right: External synapses \(v_{\textrm{ext}}=5\), transmission delay \(D = 0.08\); Middle right: External synapses \(v_{\textrm{ext}}=5\), transmission delay \(D = 0.2\); Bottom right: External synapses \(v_{\textrm{ext}}=5\), transmission delay \(D = 1\). }}
\label{fig:inhibitosc}
\end{figure}

In addition, excitatory system when the connectivity parameter \(b>0\) also exhibits various oscillation phenomena, which are shown in Figure \ref{fig:excitatory1}. We can see oscillatory solution with stable or decreasing amplitude, some of which are consistent to numerical experiments in \cite{CS2018}.

\begin{figure}
\centering
\includegraphics[width=0.46\textwidth]{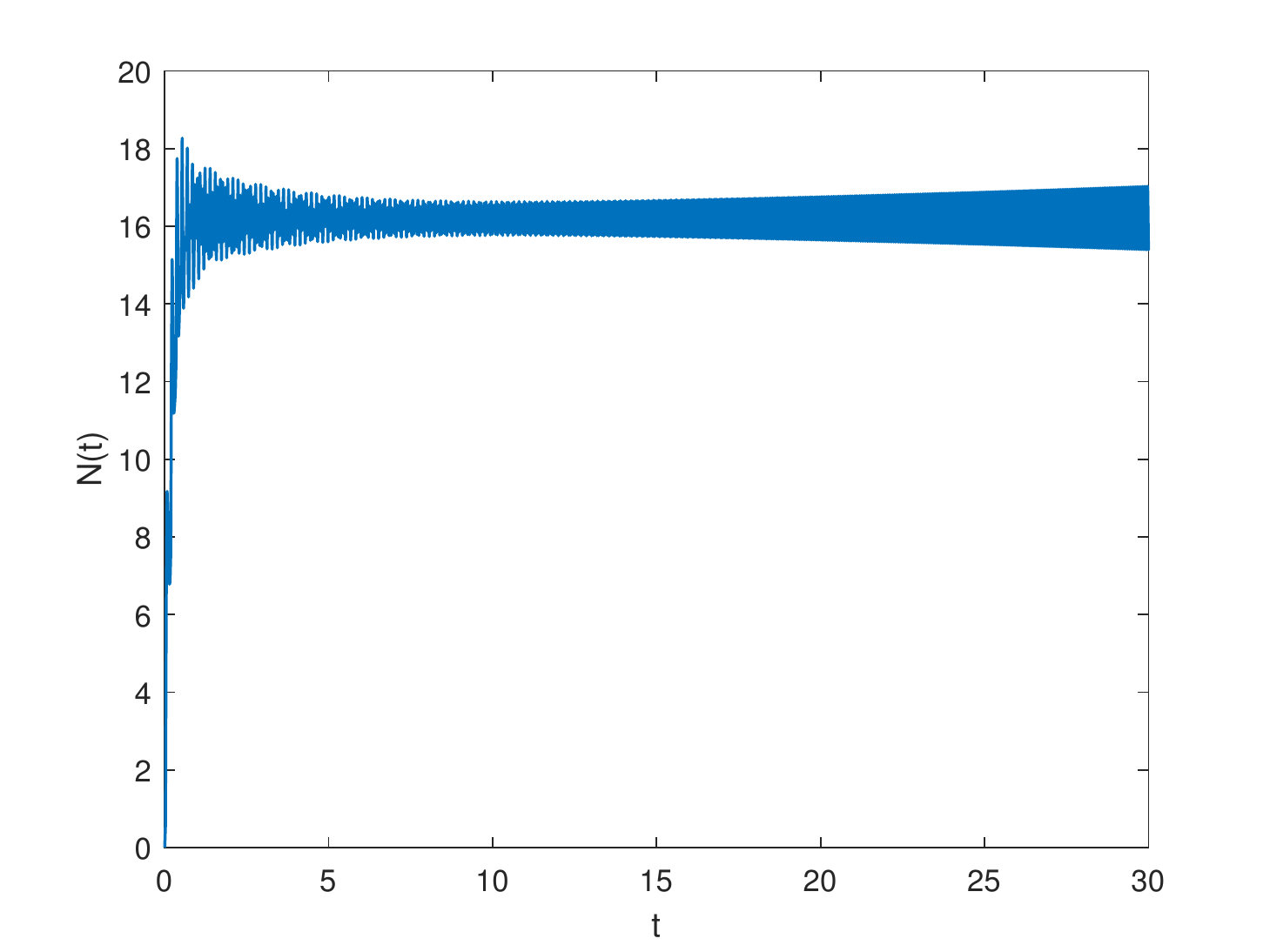}
\includegraphics[width=0.46\textwidth]{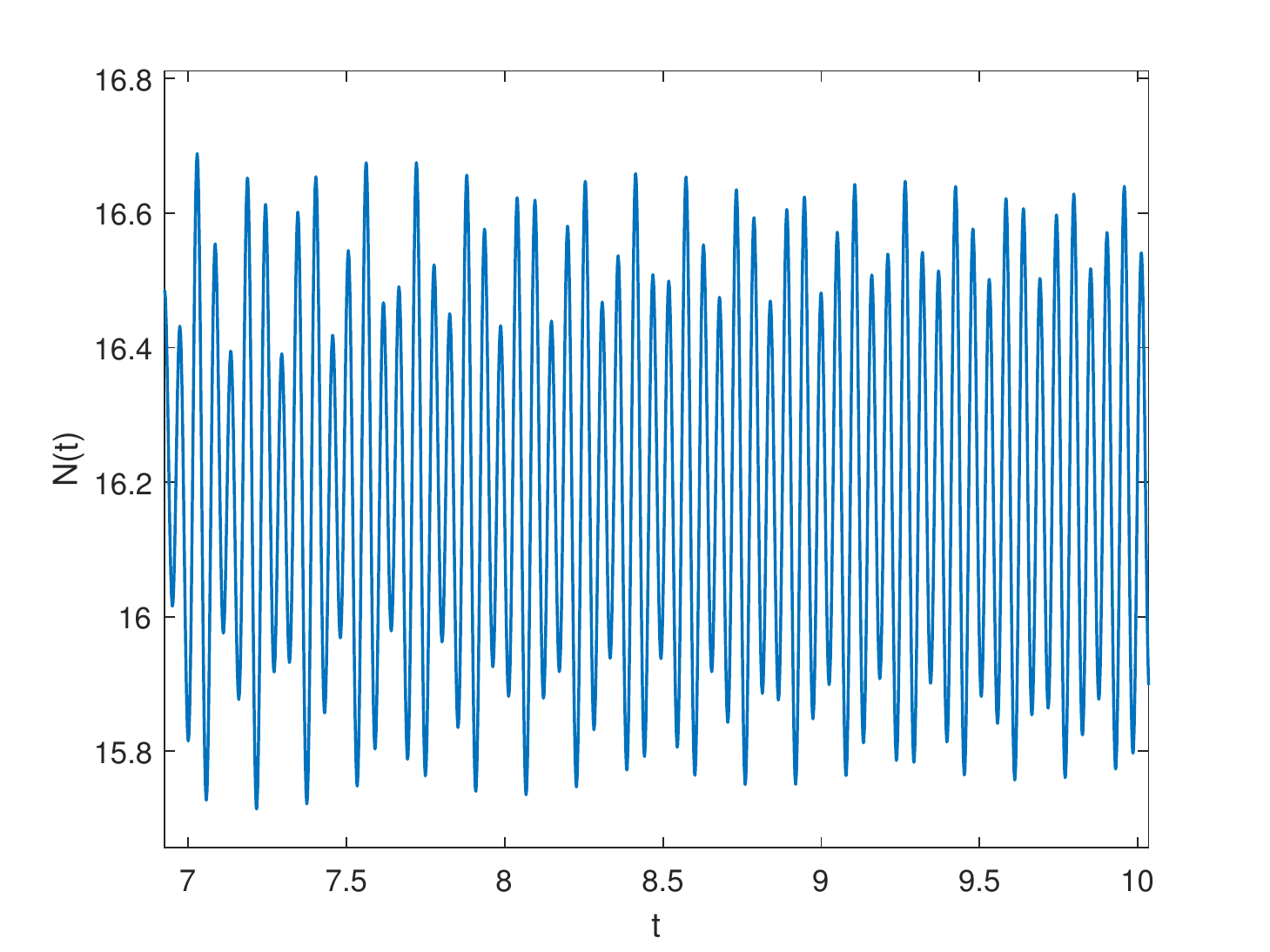}
\includegraphics[width=0.46\textwidth]{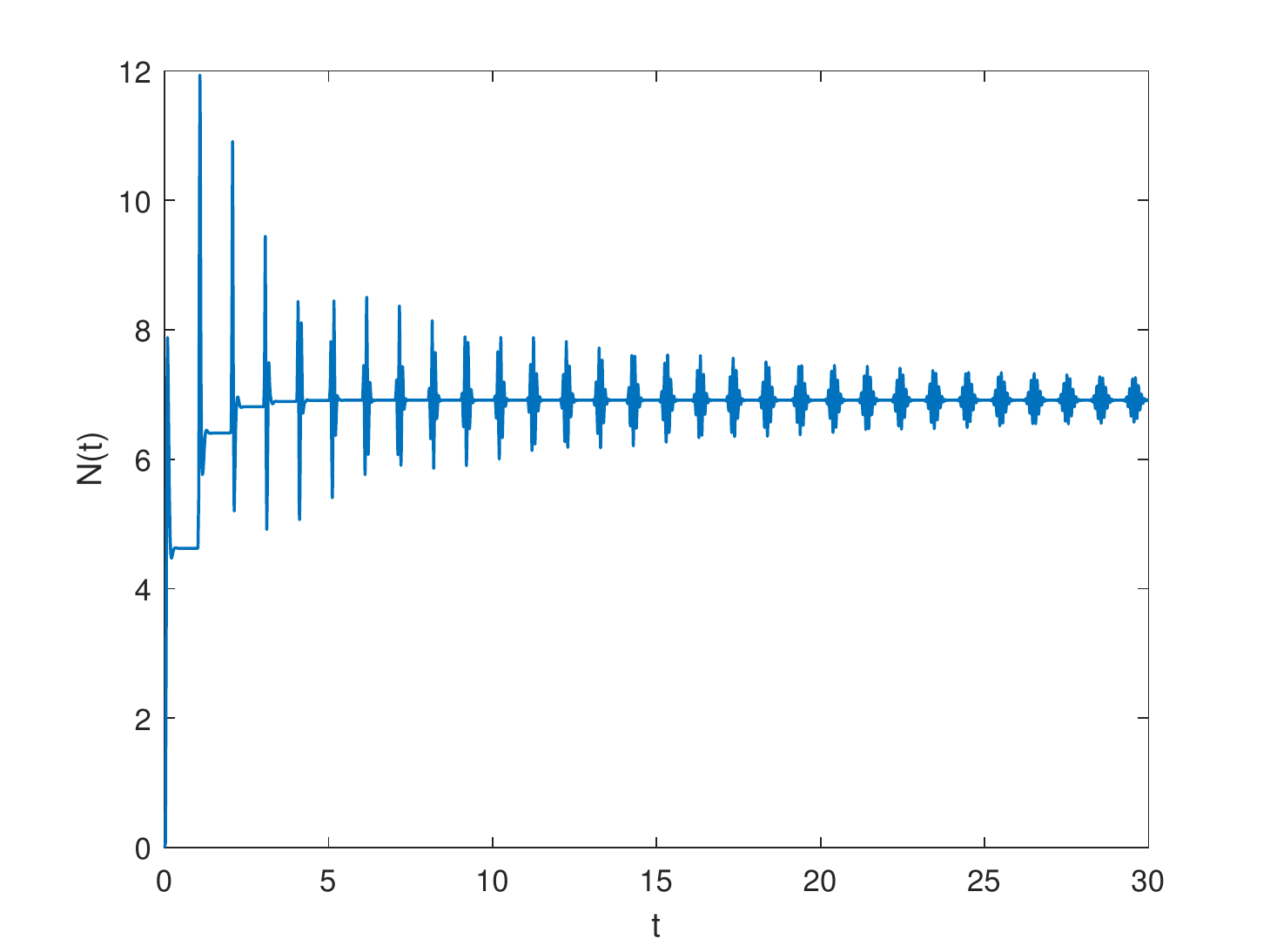}
\includegraphics[width=0.46\textwidth]{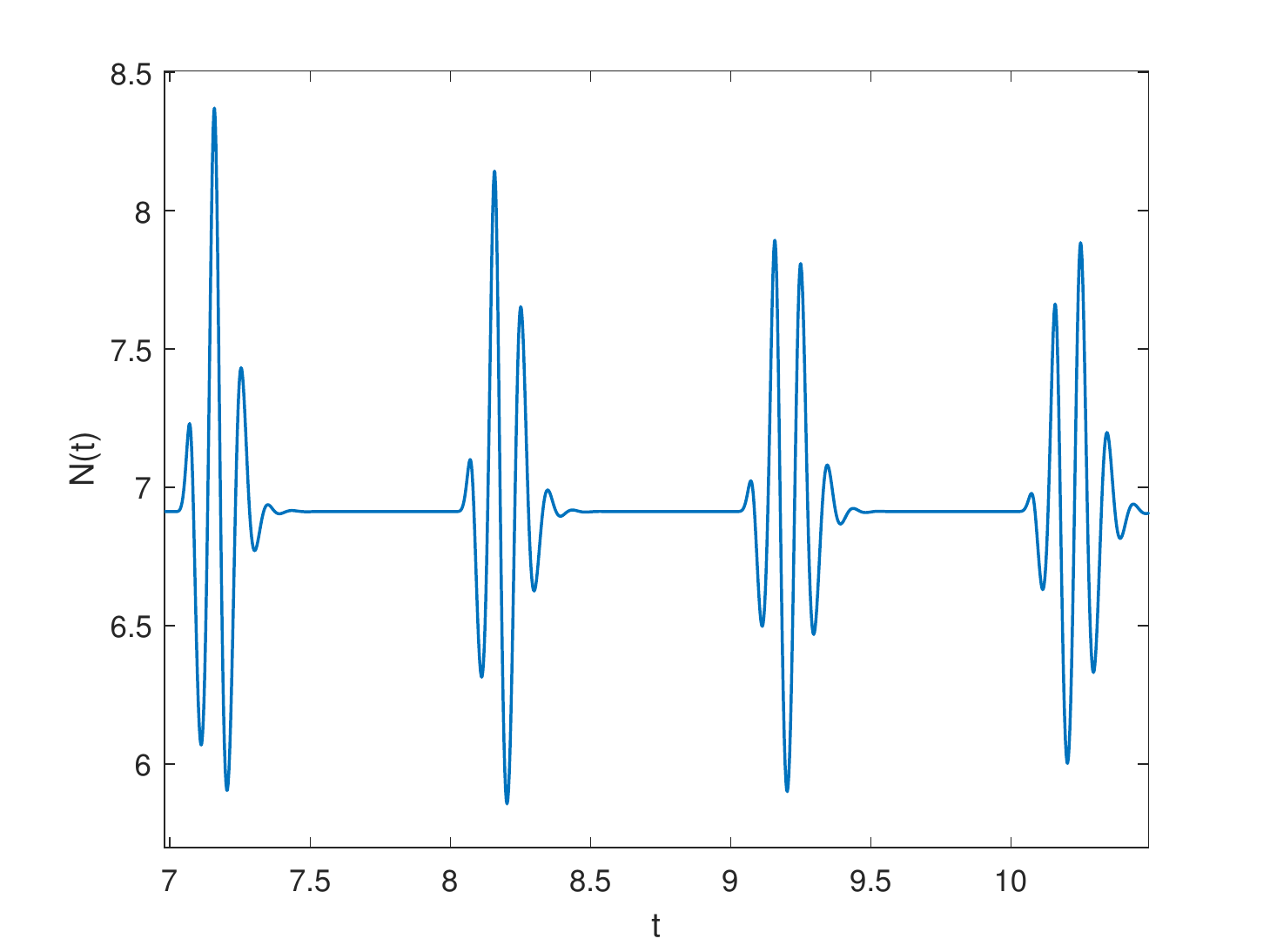}
\includegraphics[width=0.46\textwidth]{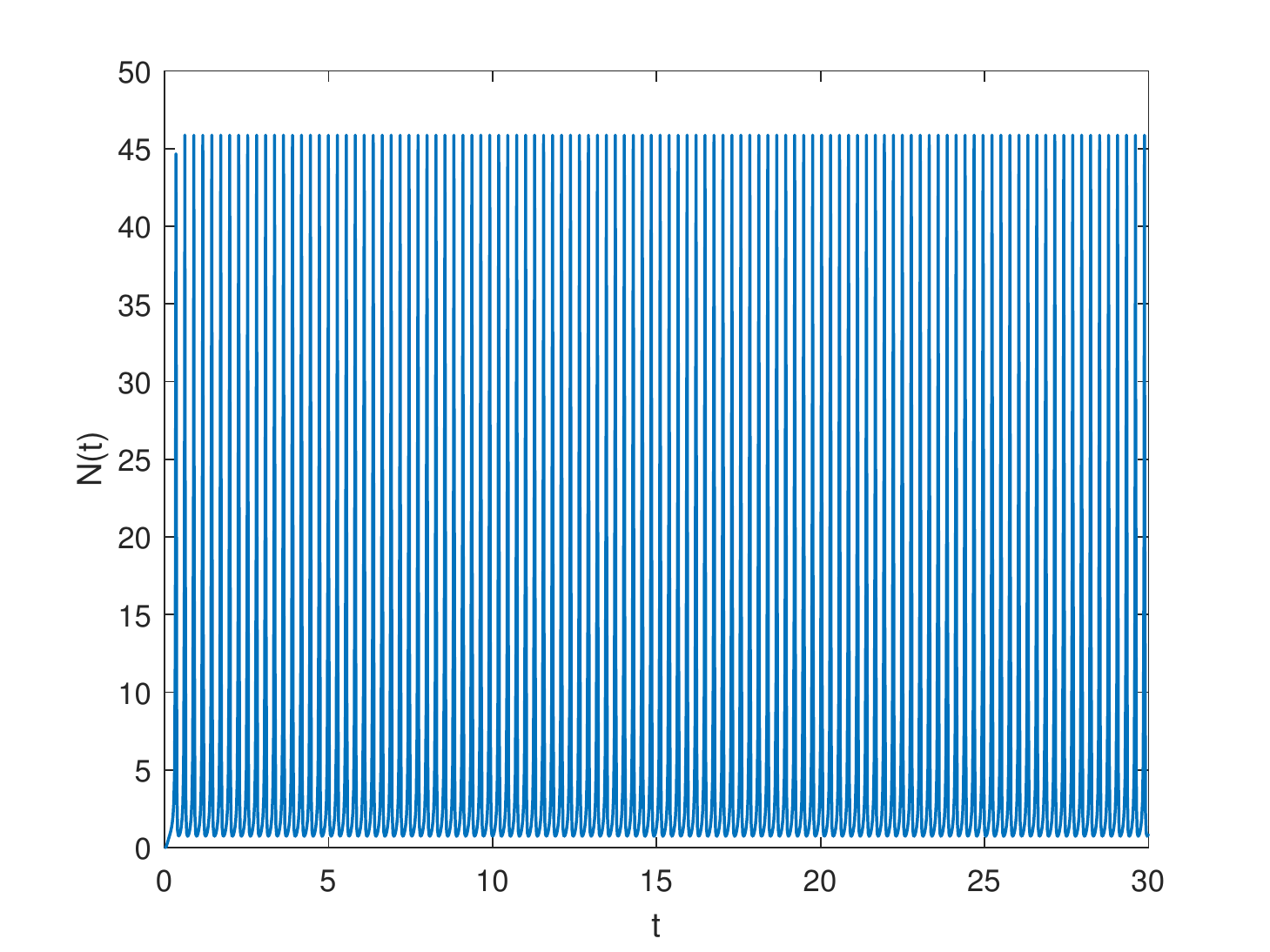}
\includegraphics[width=0.46\textwidth]{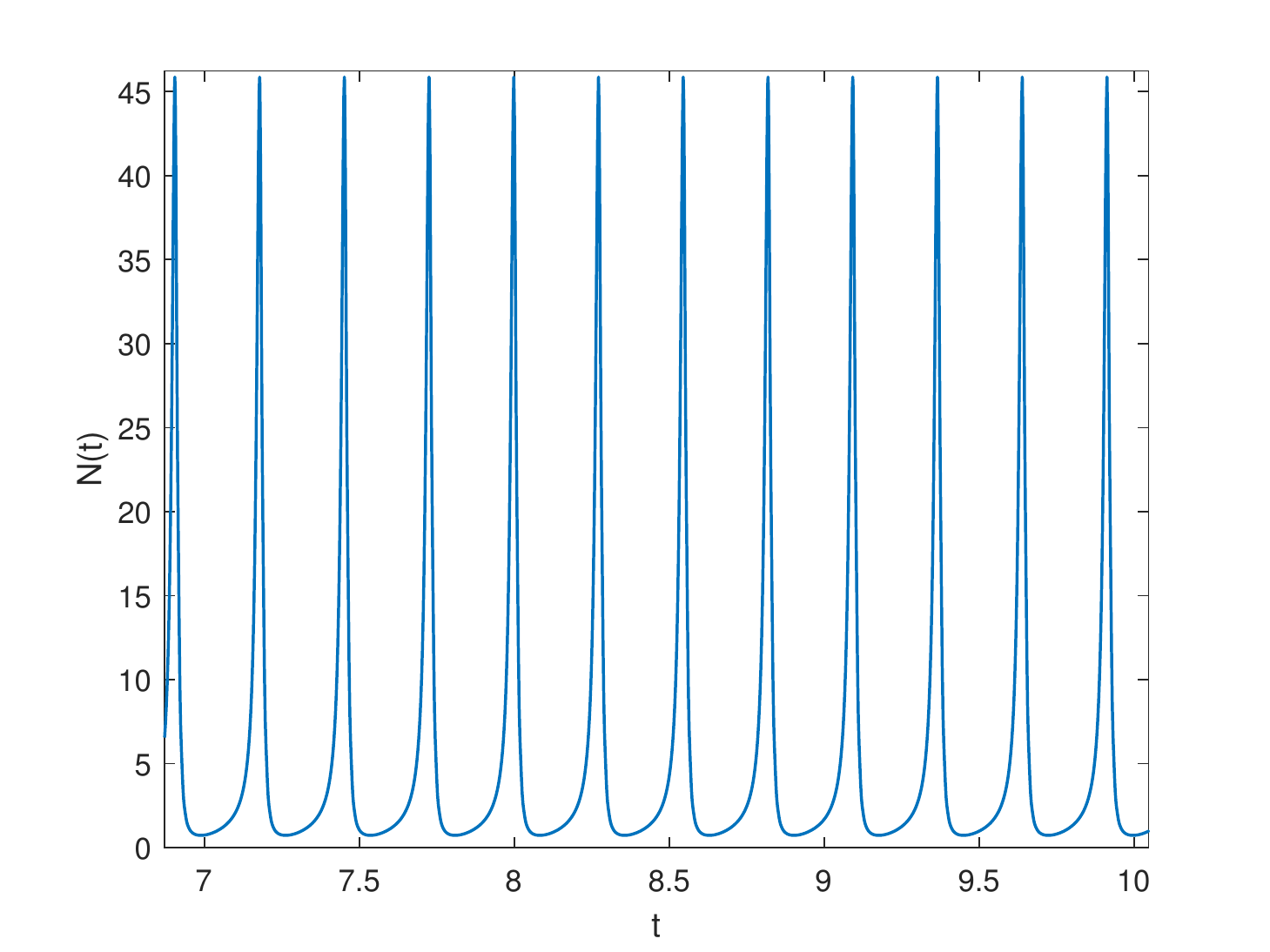}
\caption{{(Excitatory oscillatory solution) Equation parameters \(a(N(t)) \equiv 1\) and \(b>0\).  Gaussian initial condition: \(v_0=1\) and \(\sigma_0= 0.005\). We choose the spatial size \(h=\frac{2}{120}\) and the temporal size \(\tau=3\times 10^{-4}\). Top: transmission delay \(D=0.157\), connectivity parameter \(b=1.15\), external synapses \(v_{\textrm{ext}}=10\), refractory parameter \(\gamma = 0.025\), initial refractory state \(R(0)=0.2\). Top left: Bird view of the long-time firing rate evolution. Top right: A local sketch of the firing rate evolution. Middle: transmission delay \(D=1\), connectivity parameter \(b=2\), external synapses \(v_{\textrm{ext}}=10\), refractory parameter \(\gamma = 0.1\), initial refractory state \(R(0)=0.2\). Middle left: Bird view the of long-time firing rate evolution. Middle right: A local sketch of the firing rate evolution. Bottom: transmission delay \(D=0.01\), connectivity parameter \(b=3\),  external synapses \(v_{\textrm{ext}}=0\), refractory parameter \(\gamma = 0.06\), initial refractory state \(R(0)=0\). Bottom left: Bird view of the long-time firing rate evolution. Bottom right: A local sketch of the firing rate evolution.}}
\label{fig:excitatory1}
\end{figure}

In fact, compared to the case when the transmission delay \(D\) is large, oscillatory solutions with small transmission delays \(D\) are more interesting {(see figures in the bottom row of Figure \ref{fig:excitatory1} for example)}. That is because that oscillation when \(D\) is large is caused by long time delay of the neuron, rather than an intrinsic  property of the NNLIF model. On the other hand, when  \(D\) and $\gamma$ are small, the modified model can be viewed as an regularization of the original Fokker-Planck equation, and from this perspective, the numerical simulations in such scenarios may serve as useful evidences of exploring the solution structures.

\section{Conclusion}

In this work, we have studied a structure preserving numerical scheme for the Fokker-Planck equations derived from the Nonlinear Noisy Leaky Integrate-and-Fire model of neuron networks. The scheme is mass conserving, {conditionally positivity preserving} and satisfies the discrete relative entropy estimate, which is of great significance of ensuring stable and reliable numerical simulations. Besides careful convergence test, we have carried out various numerical examples, exploring different solution behaviors. In particular, the robust numerical performances with the modified model including the transmission delay and the refractory state manifest that the proposed scheme is an ideal simulator for realistic and complex neuronal network systems. In the future, we may further investigate high order extensions of the scheme, especially higher-order scheme in time, and utilize the specifically designed numerical experiments to gain insight on unknown solution properties of such Fokker-Planck equation.

\section*{Acknowledgment}
J. Hu is partially supported by NSF grant DMS-1620250 and NSF CAREER grant DMS-1654152. J. Liu is partially supported by KI-Net NSF RNMS grant No. 11-07444 and NSF grant DMS 1812573. Z. Zhou is supported by NSFC grant No. 11801016. Z. Zhou thanks Beno\^{i}t Perthame for helpful discussions.


\end{document}